\documentclass[letterpaper]{article}
\usepackage{uai2020}
\usepackage[margin=1in]{geometry}

\usepackage{times}
\usepackage{natbib}
\usepackage{graphicx}
\usepackage{amsmath,amssymb}
\usepackage{rotating}
\usepackage{multicol}
\usepackage{color}
\definecolor{lightgrey}{rgb}{0.85,0.85,0.85}

\newcommand{\bv}{\begin{array}}

\newtheorem{thm}{Theorem}
\newtheorem{dfn}{Definition}

\newtheorem{lem}{Lemma}

\newtheorem{ex}{Example}

\newenvironment{sketch}{{\noindent \em Proof sketch:} \rm }{\hfill $\Box$ }
\newenvironment{proof}{{\noindent \em Proof:} \rm }{\hfill $\Box$ }

\newtheorem{claim}{Claim}

\newcommand{\ma}{\mathcal A}
\newcommand{\wmg}{\text{WMG}}
\newcommand{\ml}{\mathcal L}
\newcommand{\mb}{\mathcal B}
\newcommand{\ms}{\mathcal S}
\newcommand{\mcp}{\mathcal P}
\newcommand{\ra}{\rightarrow}

\newcommand{\Omit}[1]{}

\newcommand{\calT}{\mathcal T}

\newcommand{\mm}{\mathcal M}

\newcommand{\others}{\text{others}}
\newcommand{\kt}{\text{KT}}

\newcommand{\ratio}{\text{Ratio}}

\newcommand{\atopr}{L_{a\succ\text{others}}}
\newcommand{\atopb}{R_{a\succ\text{others}}}

\newcommand{\size}{\text{Size}}
\newcommand{\power}{\text{Power}}

\newcommand{\lr}{\text{LR}}
\newcommand{\supp}{\text{Spt}}
\newcommand{\ext}{\text{Ext}}
\newcommand{\borda}{\text{Borda}}

\newcommand{\mallows}{\mm^\text{Ma}}
\newcommand{\condorcet}{\mm^\text{Co}}

\title{Optimal Statistical Hypothesis Testing for Social Choice}
\author{\bf Lirong Xia \thanks{This work is supported by NSF \#1453542 and \#1716333, and ONR \#N00014-17-1-2621. We thank all reviewers for helpful suggestions.} 
 \\ Rensselaer Polytechnic Institute, RPI, Troy NY 12180, USA\\ {\tt xial@cs.rpi.edu}}

\begin{document}
%

\maketitle


\begin{abstract}
We address the following question in this paper: ``{\em What are the {most} robust statistical methods for social choice?}'' By leveraging the theory of uniformly least favorable distributions in the Neyman-Pearson framework to finite models and randomized tests, we characterize {\em uniformly most powerful (UMP) tests}, which is a well-accepted statistical optimality w.r.t.~robustness, for testing whether a given alternative is the winner under Mallows' model and under Condorcet's model, respectively. 
\end{abstract}

\section{\MakeUppercase{Introduction}}

Suppose a group of seven friends want to choose restaurant $a$, $b$, or $c$ for dinner. Each person uses a ranking over the restaurants to represent his or her preferences. Three people rank $a\succ b\succ c$, three people rank $b\succ c\succ a$, and one people ranks $a\succ c\succ b$. Suppose their preferences are correlated and are based on their perception of the quality of the restaurants---the higher the quality of a restaurant, the more likely a person will rank it high. Which restaurant should they choose? 

Similar problems exist in a wide range of group decision-making scenarios such as political elections~\citep{Condorcet1785:Essai}, meta-search engines~\citep{Dwork01:Rank}, recommender systems~\citep{Ghosh99:Voting}, and crowdsourcing~\citep{Mao13:Better}. Such problems at the intersection of statistics and {social choice} can be dated back to {\em Condorcet's Jury Theorem} in the 18th century~\citep{Condorcet1785:Essai}. The Jury Theorem states that when there are two alternatives, assuming that the votes are generated i.i.d.~from a simple statistical model, then the outcome of majority voting converges to the ground truth as the number of voters goes to infinity. 

However, the Jury Theorem does not identify the {\em optimal} decision-making rule, especially when there are three alternatives or more. From a statistical point of view, defining the optimality measure is highly nontrivial and controversial. If we use likelihood of a parameter as the measure, then we may pursue the {\em likelihoodist} approach. If we view the ground truth parameter as a random variable, and use expected loss w.r.t.~the posterior distribution over the parameters as the measure, then we may pursue the {\em Bayesian} approach. If we believe that the ground truth is deterministic and unknown, and want to measure the performance of a given rule, then we may pursue the {\em frequentist} approach.\footnote{The three approaches differ in philosophy of probability and measure of rules. The same rule, for example the MLE (MAP with uniform prior for Bayesians), might be used in all three approaches due to its optimality w.r.t.~the three measures under certain conditions.}  At a high level, the frequentist approach tries to measure and design the most {\em robust} rule, as \citet{Efron2005:Bayesians} noted: 
``{\em a frequentist is a Bayesian trying to do well, or at least not too badly, against any possible prior distribution}''.

Most previous work in the literature of statistical approaches to social choice pursued either an  MLE approach or an Bayesian approach.
We are not aware of the application of a widely-applied modern frequentists' decision-making technique---{optimal statistical hypothesis testing}---to social choice. In the celebrated {\em Neyman-Pearson framework} of statistical hypothesis testing (see, e.g.~the book by~\citet{Lehmann08:Testing}), a statistical model is given and the decision-maker first chooses two non-overlapping subsets of ground truth parameters $H_0,H_1$, where $H_0$ is called the {\em null hypothesis} and $H_1$ is called the {\em alternative hypothesis}. Then the decision-maker designs a test for $H_0$ vs.~$H_1$, in the form of a {\em critical function} $f$, to make a binary decision in $\{0,1\}$ for each observed data. Here $1$ means that $H_0$ should be rejected and $0$ means that there is a lack of evidence to reject $H_0$. We note that the role of $H_0$ and $H_1$ are not the same, namely a $A$ vs.~$B$ test is different from a $B$ vs.~$A$ test. 

While many generic hypothesis testing methods can be applied, such as the {\em generalized likelihood ratio tests}~\citep{Hoeffding65:Asymptotically,Zeitouni92:When}, how to make an optimal social choice w.r.t.~frequentists' measure is still an open question.



\begin{paragraph}{\bf Our Contributions.}
We answer the question of optimal hypothesis testing for social choice by characterizing {\em uniformly most powerful (UMP)} tests for various combinations of $H_0$ and $H_1$ for winner determination under two popular models for rank data: Mallows' model and Condorcet's model.  UMP is a strong notion of optimality for hypothesis testing. 
A test $f$ is evaluated by two criteria: its {\em size} (or {\em level of significance}), which is its worst-case probability to wrongly reject $H_0$, and its {\em power}, which is its probability to correctly reject $H_0$. The power of a test is evaluated at each $h_1\in H_1$. A level-$\alpha$ test $f$ is a UMP test, if it has the highest power at every $h_1\in H_1$ among all tests whose sizes are no more than $\alpha$.

We focus on two types of tests for a given alternative $a$: the {\em non-winner tests}, where $H_0$ represents $a$ being the winner\footnote{This setting is called a ``non-winner test'' because when $H_0$ is rejected, $a$ should not be selected as the winner.}; and the {\em winner tests}, where $H_1$ represents $a$ being the winner. Our main results are summarized in Table~\ref{tab:main}.
\end{paragraph}

\begin{table}[h]
\centering
\begin{tabular}{|c|@{\small}c|@{\small}c|}
\hline
& \begin{tabular}{c}Non-winner\\ ($H_0=\{ a$ wins$\}$) \end{tabular} &\begin{tabular}{c}Winner\\ ($H_1=\{ a$ wins$\}$)\end{tabular}\\
\hline Mallows & Y\&N (Thm.~\ref{thm:winnercvsm},~\ref{thm:tvcm}) & Y\&N  (Thm.~\ref{thm:winnerbottomtopm},\ref{thm:mallowswinnernonex},\ref{thm:mallowswinnerexist}) \\

\hline Condorcet & Y\&N (Thm.~\ref{thm:winnercvsc},~\ref{thm:tvcc}) & Y  (Thm.~\ref{thm:winnerotherstopc}) \\
\hline
\end{tabular}
\caption{\small UMP tests for Mallows' model and Condorcet's model. ``Y'' in Condorcet-Winner means that for any $0<\alpha<1$, there exists a level-$\alpha$ UMP winner test for $\overline H_1$ vs.~$H_1$. ``Y\&N'' means that for some $\alpha$, no level-$\alpha$ UMP test exists for $H_0=\overline H_1$; but a UMP test  exists for all levels for some natural special cases.  \label{tab:main}}
\end{table}



For example, ``Y\&N'' in Mallows-Non-winner  in Table~\ref{tab:main}  means that for some $\alpha$, no level-$\alpha$ UMP test exists for $H_0$ vs.~$\overline H_0$, where $H_0$ consists of rankings where a given  alternative $a$ is ranked at the top. On the other hand, for some $H_1$, a level-$\alpha$ UMP test  exists for all $0<\alpha<1$. In fact, Theorem~\ref{thm:tvcm} characterizes all such $H_1$'s.

In particular, we obtained a complete characterizations of $H_1$ for which UMP non-winner tests (that is, when $H_0$ models ``$a$ wins'') exist, under Mallows' model (Theorem~\ref{thm:tvcm}) and under Condorcet's model (Theorem~\ref{thm:tvcc}).
Technically, to obtain the characterizations, we leverage the theory of {\em uniformly least favorable distributions} to finite models and randomized tests (Lemma~\ref{lem:dom}, \ref{lem:ext}, \ref{lem:ind}, \ref{lem:extdomain}). These lemmas generalize the key theorems by Reinhardt~\citet{Reinhardt1961:The-Use-of-Least} that only hold for continuous parameter space, and they might be of independent interest. 

\vspace{2mm}
{\noindent \bf Significance of results.} Our results provide the first theoretical characterization of robust social choice w.r.t.~frequentists' measure. Practically, the UMP winner tests in the Condorcet-Winner column can be used for testing whether a given alternative $a$ is a winner by appropriately setting $H_0$ while fixing $H_1$ to represent ``$a$ wins''. 

\begin{paragraph}{\bf Proof techniques.}
This paper focuses on {\em composite vs.~composite} tests, where both $H_0$ and $H_1$ contain more than one element. Many results in this paper are proved by applying Lemma~\ref{lem:hvs} (Theorem 3.8.1 and Corollary 3.8.1 in~\citep{Lehmann08:Testing}), which offers necessary and sufficient conditions for composite vs.~simple tests. However, applying Lemma~\ref{lem:hvs} is more challenging than it appears---the key is to come up with a {\em uniformly least favorable distribution} that satisfies the conditions in Lemma~\ref{lem:hvs} {\em for all elements in $H_1$}, and such distribution is not guaranteed to exist. As we show later in the paper, such distributions indeed exist for non-winner tests for Mallows' model and Condorcet's model respectively, and it is non-trivial to verify that they satisfy conditions in Lemma~\ref{lem:hvs}. In fact, to this end, we proved new properties (Lemma~\ref{lem:CC'} and Lemma~\ref{lem:bordadom} in the appendix) about Mallows' model and new general theorems (Lemma~\ref{lem:ind} and~\ref{lem:extdomain}) that can be applied to Condorcet's model, which might be of independent interest. 
\end{paragraph}

\begin{paragraph}{\bf Related Work and Discussions.}
\citet{Marden95:Analyzing} applied the Neyman-Pearson Lemma (Lemma~\ref{lem:NP}) for simple vs.~simple tests under Mallows, as illustrated in Example~\ref{ex:mallows-simple-vs-simple}. Most previous work in statistical approaches in social choice focused on extending the Condorcet Jury Theorem and proving asymptotic results~\citep{Gerlinga05:Information,Nitzan17:Collective}. Previous work focused on using commonly-studied voting rules designed for elections~\citep{Conitzer05:Common,Caragiannis2016:When}, maximum likelihood estimators~\citep{Conitzer05:Common,Xia11:Maximum}, or Bayesian estimators~\citep{Young88:Condorcet,Procaccia12:Maximum,Pivato13:Voting,Elkind14:Electing,Azari14:Statistical,Xia2016:Bayesian}. We are not aware of a previous work on UMP tests for deciding whether a given alternative wins or not in social choice context. 


Compared to previous MLE and Bayesian approaches to social choice, optimal rules characterized in this paper are more robust because it offers the best worst-case guarantee against an adversary who controls the ground truth parameter. As in the general Bayesian vs.~Frequentists debate, this does not mean that one approach is better than another, because the measures  of performance are different. 

\end{paragraph}
\section{\MakeUppercase{Preliminaries}}

Let $\ma=\{a_1,\ldots,a_m\}$ denote a set of $m\ge 2$ {\em alternatives} and let $\ml(\ma)$ denote the set of all linear orders over $\ma$. Let $n$ denote the number of agents. Each agent's preferences are represented by a linear order in $\ml(\ma)$. We often use $V=[a\succ b\succ \cdots]$ to denote a ranking, and write $a\succ_V b$ if $a$ is preferred to $b$ in $V$. Let $P_n$ denote the collection of $n$ agents' votes, called an {\em ($n$-)profile}. For any profile $P$ and any pair of alternatives $a,b$, we let $P(a\succ b)$ denote the number of votes in $P$ where $a$ is preferred to $b$. 

The {\em weighted majority graph (WMG)} of $P$, denoted by $\wmg(P)$, is a directed weighted graph where the weight $w_P(a\succ b)$ on any edge $a\ra b$ is $w_P(a\succ b)=P(a\succ b)-P(b\succ a)$. By definition $w_P(a\succ b)=-w_P(b\succ a)$. For example, the WMG of the profile $P_7$ of seven linear orders mentioned in the beginning of Introduction has weights $w_{P_7}(a\succ b)= w_{P_7}(a\succ c)=1$ and $w_{P_7}(b\succ c)=5$.

\begin{paragraph}{\bf Statistical Models for Rank Data.}
A statistical model $\mm=(\ms,\Theta,\vec\pi)$ has three parts: the {\em sample space} $\ms$, which is composed of all possible data; the {\em parameter space} $\Theta$; and the probability distributions $\vec \pi=\{\pi_\theta:\theta\in \Theta\}$. If both $\ms$ and $\Theta$ contain finitely many elements, then we call $\mm$ a {\em finite model}. 
 For any pair of linear orders $V,W$ in $\ml(\ma)$, let $\kt(V,W)$ denote the {\em Kendall-tau distance},  which is the total number of pairwise disagreements between $V$ and $W$. Formally, 
$$
\kt(V,W) = \#\left\{\begin{array}{r}
\{a,b\}\subseteq \ma: [a\succ_V b\text{ and }b\succ_W a] \\
\text{ or } [b\succ_V a\text{ and }a\succ_W b]
\end{array}\right\}
$$


\begin{dfn}\label{dfn:modelmallows} {\bf (Mallows' model with fixed dispersion~\citep{Mallows57:Non-null})} Given the dispersion $0<\varphi<1$,  {\em Mallows' model} is denoted by $\mallows=(\ml(\ma)^n,\ml(\ma),\vec\pi)$, where $n$ linear orders are i.i.d.~generated, the parameter space is $\ml(\ma)$ and for any $V,W\in\ml(\ma)$,  $\pi_W(V)=\frac{1}{Z}\varphi^{\kt(V, W)}$, where $Z$ is the normalization factor.
\end{dfn}

Condorcet's model differs from Mallows' model by allowing ties in the ground truth and in data. Formally, let $\mb(\ma)$ denote the set of all irreflexive, antisymmetric, and
total {\em binary relations} over $\ma$. We have $\ml(\ma)\subseteq \mb(\ma)$ and the
Kendall-tau distance is extended to  $\mb(\ma)$ by counting the number of pairwise disagreements. 

\begin{dfn}\label{dfn:modelcondorcet}{\bf (Condorcet's model for binary relations with fixed dispersion)} Given the dispersion $0<\varphi<1$, {\em Condorcet's model} is denoted by $\condorcet=(\mb(\ma)^n,\mb(\ma),\vec\pi)$, where the parameter space is $\mb(\ma)$ and for any $W\in\mb(\ma)$ and $V\in\mb(\ma)$,  $\pi_W(V)=\frac{1}{Z}\varphi^{\kt(V, W)}$, where $Z$ is the normalization factor. \end{dfn}

In classical Condorcet's model~\citep{Condorcet1785:Essai,Young88:Condorcet}, the sample space consists of linear orders and the parameter space consists of binary relations. The model in Definition~\ref{dfn:modelcondorcet} is a variant of Condorcet's model, where the sample space consists of binary relations. In other words, each agent is allowed to use a binary relation to represent his or her preferences---transitivity is not required as in classical Condorcet's model or Mallows' model.

\end{paragraph}

\begin{paragraph}{\bf Statistical Hypothesis Testing: The Neyman-Pearson Framework.}
Given a statistical model $\mm=(\ms,\Theta,\vec \pi)$, the decision-maker first chooses two non-overlapping subsets of parameters $H_0,H_1\subseteq \Theta$, where $H_0$ is called the {\em null hypothesis} and $H_1$ is called the {\em alternative hypothesis}. The goal of hypothesis testing is to decide whether the ground truth parameter is in $H_0$ ({\em retaining} the null hypothesis) or in $H_1$ ({\em rejecting} the null hypothesis), based on the observed data $P\in \ms$.  To simplify notation, we let $0$ denote retain and let $1$ denote reject. A test is characterized by a (randomized) {\em critical function} $f:\ms\ra [0,1]$ such that for any $P\in\ms$, with probability $f(P)$ the outcome of testing is $1$ ({reject}). When $H_0$ (or $H_1$) contains a single parameter, it is called a {\em simple} hypothesis; otherwise it is called a {\em composite} hypothesis.

A test $f$ is often evaluated by its {\em size} and {\em power}. The size of $f$ is the maximum probability for $f$ to wrongly outputs $1$ when the ground truth is in $H_0$ (such cases are called Type I errors or false positives), where the max is taken over all parameters in $H_0$. More precisely, for any $h_0\in H_0$, we let $\size(f,h_0)=E_{P\sim \pi_{h_0}}f(P)$, and $\size(f)=\sup_{h_0\in H_0}\size(f,h_0)$. If the size of $f$ is $\alpha$, then $f$ is called a {\em level-$\alpha$} test. For any $h_1\in H_1$, the power of $f$ at $h_1$ is the probability that $f$ correctly outputs $1$ when the ground truth is $h_1$. More precisely, we let $\power(f,h_1)=E_{P\sim \pi_{h_1}}f(P)$, where the expectation is take over randomly generated profiles from $\pi_{h_1}$. 
We would like  a test $f$ to have low size and high power, but often tradeoffs must be made.

\begin{ex}
Let $\mallows$ denote a Mallows' model with $m=3$ and $n=1$. Let $\ma=\{1,2,3\}$, $h_1=[1\succ 2\succ 3]$, and let $H_0$ denote the other rankings. Let $f$ be a test where $f(1\succ 2\succ 3)=1,f(2\succ 1\succ 3)=f(1\succ 3\succ 2)=0.5$, and $f$ outputs $0$ for all other rankings. We have $\size(f)=\size(f,2\succ 1\succ 3)=\size(f,1\succ 3\succ 2)=(0.5+\varphi+0.5\varphi^2)/Z$, where $Z$ is the normalization factor. $\power(f,1\succ 2\succ 3)=(1+\varphi)/Z$.
\end{ex}

Given a statistical model $\mm$, $H_0$, a parameter $h_1\not\in H_0$, and $0< \alpha<1$, a {\em level-$\alpha$ most powerful test} $f_\alpha$ is a test with the highest power among all tests whose size is no more than $\alpha$. 
For finite $H_0$, a most powerful test always exists and may not be unique.
For composite $H_1$, it is possible that for different $h_1\in H_1$, the most powerful tests are different. If there exists a level-$\alpha$ test $f_\alpha$ that is most powerful for all $h_1\in H_1$, then $f$ is called a level-$\alpha$ {\em uniformly most powerful (UMP) test} for $H_0$ vs.~$H_1$. UMP is a strong notion of optimality and a UMP test may not exists. 


For simple $H_0$ vs.~simple $H_1$, that is, $|H_0|=|H_1|=1$, the fundamental lemma of Neyman and Pearson characterizes the most powerful tests as {\em likelihood ratio tests}, defined as follows. 
\begin{dfn}[\bf Likelihood ratio test]
\label{dfn:lrtest}Given a model $\mm$ and $0< \alpha< 1$. For any $h_0, h_1\in \Theta$ with $h_0\ne h_1$ and any $P\in\ms$, we let $\ratio_{h_0,h_1}(P)=\frac{\pi_{h_1}(P)}{\pi_{h_0}(P)}$ denote the {\em likelihood ratio} of $P$ and let
$$\lr_{\alpha,h_0,h_1}(P)=\left\{\begin{array}{cc}1&\text{if }\ratio_{h_0,h_1}(P)>k_\alpha\\
0&\text{if }\ratio_{h_0,h_1}(P)<k_\alpha\\
\gamma_\alpha&\text{if }\ratio_{h_0,h_1}(P)=k_\alpha
\end{array}\right.,$$
denote the {\em level-$\alpha$ likelihood ratio test}, where $k_\alpha\ge 0$ and $\gamma_\alpha$ are chosen such that $\size(\lr_{\alpha,h_0,h_1})=\alpha$.
\end{dfn}
\begin{lem}\label{lem:NP} {\bf (The Neyman-Pearson Lemma, see e.g.~\citep{Lehmann08:Testing})} For any simple vs.~simple test ($h_0$ vs.~$h_1$) and any $0<\alpha< 1$, the likelihood ratio test $\lr_{\alpha,h_0,h_1}$ is a level-$\alpha$ most powerful test. Moreover, any most powerful test must agree with $\lr_{\alpha,h_0,h_1}$ except on $P\in\ms$ with $\ratio_{h_0,h_1}(P)=k_{\alpha}$.
\end{lem}

\Omit{The Neyman-Pearson lemma has the following intuitive explanation in the context of the production problem. Now there is a single worker and a single buyer. The optimal plan is to start with the worker's most cost-effective items---those with highest dollar per hour ratio: $\ratio(P)=\frac{\pi_{h_1}(P)}{\pi_{h_0}(P)}$, and then move on to the next most cost-effective item, until her limit $\alpha$ is reached. This corresponds to the most powerful test guaranteed by the Neyman-Pearson lemma. For different types of items with the same cost-effectiveness, the worker can choose any item to produce. This corresponds to the cases with $\ratio(P,h_0,h_1)=k_\alpha$.
}

\begin{ex}\label{ex:mallows-simple-vs-simple}
\rm Given a  Mallows' model. Let $H_0=\{h_0\}$ and $H_1=\{h_1\}$. For any $n$-profile $P_n$, we have $\ratio(P_n)=\frac{\varphi^{\kt(P_n,h_1)}}{\varphi^{\kt(P_n,h_0)}}=\varphi^{\kt(P_n,h_1)-\kt(P_n,h_0)}$. Therefore, it follows from the Neyman-Pearson lemma that for any $0< \alpha< 1$, there exist $K_\alpha$ and $\Gamma_\alpha$ such that the following test $f_\alpha$ is a level-$\alpha$ most powerful test: for any $n$-profile $P_n$, 
$$f_{\alpha}(P_n)=\left\{\begin{array}{ll}1&\text{if }\kt(P_n,h_0)-\kt(P_n,h_1)>K_\alpha\\
0&\text{if }\kt(P_n,h_0)-\kt(P_n,h_1)<K_\alpha\\ 
\Gamma_\alpha&\text{if }\kt(P_n,h_0)-\kt(P_n,h_1)=K_\alpha
\end{array}\right..$$
\hfill$\blacksquare$
\end{ex}

For composite $H_0$ vs simple $H_1=\{h_1\}$, a generalization of the Neyman-Pearson lemma exists. The idea is to use a distribution $\Lambda$ over $H_0$ to compress $H_0$ into a ``combined'' parameter, defined as follows. 

\begin{dfn}
For any $\mm=(\ms,\Theta,\vec\pi)$, any $H_0\subseteq \Theta$, and any $h_1\in(\Theta\setminus H_0)$. Let $\Lambda$ denote a distribution over $H_0$ whose support set is denoted by $\supp(\Lambda)$, and let $h_0^\Lambda$ denote a new parameter whose distribution over $\ms$ is the probabilistic mixture of $\{\pi_{h_0}:h_0\in\Theta\}$ according to $\Lambda$. For any $0< \alpha< 1$ and any $P\in\ms$, 
\begin{itemize}
\item let $\ratio_{\Lambda,h_1}(P) =\frac{\pi_{h_1}(P)}{\sum_{h_0\in H_0}\Lambda(h_0)\pi_{h_0}(P)}$, and 
\item let $\lr_{\alpha,\Lambda,h_1}(P)$ denote the {likelihood ratio test} for $h_0^\Lambda$ vs.~$h_1$ as in Definition~\ref{dfn:lrtest}. 
\end{itemize}
\end{dfn}
The following lemma states that $\lr_{\alpha,\Lambda,h_1}$ is a most powerful test for $H_0$ vs.~$h_1$ iff two  conditions are satisfied.
\begin{lem}\label{lem:hvs} {\bf (Theorem 3.8.1 and Corollary 3.8.~by \citet{Lehmann08:Testing})}
For composite vs.~simple test ($H_0$ vs.~$h_1$) and any distribution $\Lambda$ over $H_0$, the likelihood ratio test $\lr_{\alpha,\Lambda,h_1}$ is a level-$\alpha$ most powerful test  if and only if the following two conditions are satisfied. 

(i) For any $h_0^*\in \supp(\Lambda)$,  $\size(\lr_{\alpha,\Lambda,h_1},h_0^*)=\alpha$. 

(ii) For any $h_0\in H_0$, $\size(\lr_{\alpha,\Lambda,h_1},h_0)\le \alpha$. 

Moreover, if there is no $P\in\ms$ with $\ratio_{\Lambda,h_1}(P)=k_\alpha$, then $\lr_{\alpha,\Lambda,h_1}$ is the unique level-$\alpha$ most powerful  test.
\end{lem}

The distribution $\Lambda$ in Lemma~\ref{lem:hvs} is called a {\em least favorable distribution}. If $|\supp(\Lambda)|=1$, then $\Lambda$ is called a {\em deterministic least favorable distribution}. If $\Lambda$ is a least favorable distribution for all levels of significance $0<\alpha< 1$, then it is called a {\em uniformly least favorable distribution}~\citep{Reinhardt1961:The-Use-of-Least}. 


\end{paragraph}

\section{\MakeUppercase{Test Setup and Basic Lemmas}}
We first introduce two types of hypothesis tests for choices. 
Given an alternative $a$, for Mallows' model we define $\atopr=\{V\in\ml(\ma):\forall b\in\ma, a\succ_V b\}$; similarly, for Condorcet's model we define $\atopb=\{V\in\mb(\ma):\forall b\in\ma, a\succ_V b\}$. $\atopr$ and $\atopb$ naturally correspond to $a$ being ranked at the top in the the ground truth in Mallows' model and in Condorcet's model, respectively.

\begin{dfn}[(Non-)Winner Tests] Given an alternative $a$, in a {\em non-winner} test for Mallows' model, we let $H_0=\atopr$; and in a {\em winner} test for Mallows' model, we let $H_1=\atopr$. 

 Given an alternative $a$,  in a {\em non-winner} test for Condorcet's model, we let $H_0=\atopb$;   and in a {\em winner} test for Condorcet's model, we let $H_1=\atopb$.
\end{dfn}

The rationale behind the naming of ``non-winner" and ``winner'' is the following. Because $H_0$ is often chosen as the devil's advocate and the goal of testing is often to reject $H_0$, when setting $H_0=\atopr$ under Mallows' model, we are hoping to reject $H_0$, which means that $a$ is not the winner. 
We note  that the decision-maker still needs to specify $H_1$ in a non-winner test and specify $H_0$ in a winner test. Various natural choices of $H_1$ or $H_0$ will be explored in Section~\ref{sec:mallows} and Section~\ref{sec:condorcet}. 

We now  present two general lemmas on least favorable distributions that will be frequently used in this paper. 
For any model $\mm=(\ms,\Theta,\vec\pi)$, any composite vs.~simple test ($H_0$ vs.~$h_1$), any distribution $\Lambda$ over $H_0$, and any $h_0\in H_0$, we define a random variable $X_{h_0}^\Lambda:\ms\ra \mathbb R$ such that for any $P\in\ms$, $\Pr(P)=\pi_{h_0}(P)$ and $X_{h_0}^\Lambda(P)=\log\ratio_{\Lambda,h_1}(P)$. A random variable $X$ {\em weakly  first-order stochastically dominates} ({\em weakly dominates} for short) another random variable $Y$, if for all $p\in \mathbb R$, $\Pr(X\ge p)\ge \Pr(Y\ge p)$.

\begin{lem}\label{lem:dom} $\Lambda$ is a uniformly least favorable distribution for $H_0$ vs.~$h_1$ if and only if for any $h_0^*\in\supp(\Lambda)$ and any $h_0\in H_0$, $X_{h_0^*}^\Lambda$ weakly dominates $X_{h_0}^\Lambda$.
\end{lem}
\begin{proof} To simplify notation we let $\lr_\alpha$ and $\ratio$ to denote $\lr_{\alpha,\Lambda,h_1}$ and $\ratio_{\Lambda,h_1}$, respectively. For any $0< \alpha< 1$ and any $h_0\in H_0$, we have 
\begin{align}
&\size(\lr_{\alpha},h_0)\notag=\sum\nolimits_{P\in\ms:\ratio(P)>k_\alpha}\pi_{h_0}(P)\\
&+\gamma_\alpha \sum\nolimits_{P\in\ms:\ratio(P)=k_\alpha}\pi_{h_0}(P)\notag\\
=& \Pr(X_{h_0}^\Lambda>\log k_\alpha)+\gamma_\alpha \Pr(X_{h_0}^\Lambda=\log k_\alpha)\\
=&(1-\gamma_\alpha)\Pr(X_{h_0}^\Lambda> \log k_\alpha)+\gamma_\alpha \Pr(X_{h_0}^\Lambda\ge \log k_\alpha)\notag\\
=&  (1-\gamma_\alpha)\lim_{x\ra \log k_\alpha^-}\Pr(X_{h_0}^\Lambda\ge  x)+ \gamma_\alpha \Pr(X_{h_0}^\Lambda\ge k_\alpha)\notag\label{eq:sizedom}
\end{align}

The ``if'' direction: for any $h_0\in H_0$ and any $h_0^*\in\supp(\Lambda)$, because $X_{h_0^*}^\Lambda$ weakly dominates $X_{h_0}^\Lambda$, we have that for any $x\in\mathbb R$, $\Pr(X_{h_0^*}^\Lambda\ge x)\ge \Pr(X_{h_0}^\Lambda\ge x)$. It follows from (1) that $\size(\lr_{\alpha},h_1),h_0^*) \ge \size(\lr_{\alpha},h_0)$. By Lemma~\ref{lem:hvs}, $\lr_{\alpha}$ is a level-$\alpha$ most powerful test. Therefore $\Lambda$ is a uniformly least favorable distribution.

The ``only if'' direction: suppose for the sake of contradiction that this is not true. Let $h_0\in H_0$ and $h_0^*\in\supp(\Lambda)$ be such that $X_{h_0^*}^\Lambda$ does not weakly dominate $X_{h_0}^\Lambda$. It follows that there exists $x\in\mathbb R$ such that $\Pr(X_{h_0^*}^\Lambda\ge x)<\Pr(X_{h_0}^\Lambda\ge x)$. Let $\alpha=\Pr(X_{h_0^*}^\Lambda\ge x)$. Because $\Lambda$ is uniformly least favorable, the size of $\lr_{\alpha}$ must be $\alpha$, where $k_\alpha=2^x$ and $\gamma_\alpha=1$. By Lemma~\ref{lem:hvs}, $\Pr(X_{h_0}^\Lambda\ge x)=\size(\lr_{\alpha},h_0)\ge \size(\lr_{\alpha},h_0^*)=\Pr(X_{h_0^*}^\Lambda\ge x)$, which is a contradiction.
\end{proof}

\begin{ex}\label{ex:umpm3} \rm Let $\mm$ denote a Mallows' model with $m=3$ and $n=1$. Let $\ma=\{1,2,3\}$, $h_1=[1\succ 2\succ 3]$ and let $\Lambda$ denote the uniform distribution over $\{[2\succ 1\succ 3], [1\succ 3\succ 2]\}$. We will apply Lemma~\ref{lem:dom} to prove that $\Lambda$ is a uniformly least favorable distribution for testing $H_0=(\ml(\ma)-\{[1\succ 2\succ 3]\})$ vs.~$[1\succ 2\succ 3]$. The likelihood ratios of all rankings are summarized in Table~\ref{tab:ratio} in the increasing order.
\begin{table}[htp]
\centering
\begin{tabular}{|c|c|c|c|}
\hline $V$ & $3\succ 2\succ 1$ & others & $1\succ 2\succ 3$\\
\hline $\ratio_{\Lambda,1\succ 2\succ 3}(V)$:&$\varphi$&$\frac{2\varphi}{1+\varphi^2}$&$\frac{1}{\varphi}$\\
\hline
\end{tabular}
\caption{Likelihood ratios.\label{tab:ratio}\vspace{-6mm}}
\end{table}

For any $h_1\in H_0$, $X_{h_0}^\Lambda$ takes three values: $\log \frac{1}{\varphi}$, $\log \frac{2\varphi}{1+\varphi^2}$, and $\log {\varphi}$. The probabilities for the five random variables taking these three values  are summarized in Table~\ref{tab:rv3}.
\begin{table}[htp]
\centering
\begin{tabular}{|c|c|c|c|}
\hline  &$\log \varphi$&$\log \frac{2\varphi}{1+\varphi^2}$&$\log \frac{1}{\varphi}$\\
\hline $X_{1\succ 3\succ 2}^\Lambda$ and $X_{2\succ 1\succ 3}^\Lambda$& $\frac{\varphi^2}{Z}$&  $\frac{1+\varphi+\varphi^2+\varphi^3}{Z}$ & $\frac{\varphi}{Z}$\\
\hline $X_{2\succ 3\succ 1}^\Lambda$ and  $X_{3\succ 1\succ 2}^\Lambda$& $\frac{\varphi}{Z}$&  $\frac{1+\varphi+\varphi^2+\varphi^3}{Z}$ & $\frac{\varphi^2}{Z}$\\
\hline $X_{3\succ 2\succ 1}^\Lambda$& $\frac{1}{Z}$&  $\frac{2(\varphi+\varphi^2)}{Z}$ & $\frac{\varphi^3}{Z}$\\
\hline
\end{tabular}
\caption{$X_{h_0}^\Lambda$ for all $h_0\in H_0$, where $Z$ is the normalization factor.\label{tab:rv3}}
\end{table}

Because $0<\varphi<1$, it is not hard to verify that $X_{1\succ 3\succ 2}^\Lambda$ and $X_{2\succ 1\succ 3}^\Lambda$ weakly dominate other random variables. By Lemma~\ref{lem:dom}, $\Lambda$ is a uniformly least favorable distribution. \hfill$\blacksquare$
\end{ex}

Our second lemma states that if we can find a deterministic uniformly least favorable distribution for $n=1$, then it is also uniformly least favorable for the same statistical model with $n\ge 2$ i.i.d.~samples.

\begin{lem}\label{lem:ext} Suppose $\Lambda$ is a deterministic uniformly least favorable distribution for composite vs.~simple test ($H_0$ vs.~$h_1$) under $\mm=(\ms,\Theta,\vec\pi)$. Then for any $n\in \mathbb N$, $\Lambda$ is also a uniformly least favorable distribution for testing $H_0$ vs.~$h_1$ under $\mm=(\ms^n,\Theta,\vec\pi)$ with $n$ i.i.d.~samples.
\end{lem}
 {All missing proofs can be found in the appendix.}
\Omit{
\begin{proof}  Let $\supp(\Lambda)=\{h_0^*\}$.  For any $n\in\mathbb N$ and any $h_0\in H_0$, we define a random variable $X_{n,h_0}:\ms^n\ra \mathbb R$, where for any $P_n\in\ms^n$, $\Pr(P_n)=\pi_{h_0}(P_n)=\prod_{V\in P_n}\pi_{h_0}(V)$, and $X_{n,h_0}(P_n)=\log \ratio_{h_0^*,h_1}$. It follows that $X_{n,h_0}=\underbrace{X_{h_0}+ X_{h_0}+ \cdots+ X_{h_0}}_n$. By Lemma~\ref{lem:dom},  for any $h_0\in H_0$, $X_{h_0^*}$ weakly dominates $X_{h_0}$. Because first-order stochastic dominance is preserved under convolution~\citep{Deelstra14:Risk}, we have that $X_{n,h_0^*}$ weakly dominates $X_{n,h_0}$. The lemma follows after applying Lemma~\ref{lem:dom}.\end{proof}

\vspace{2mm}
\noindent{\bf Remarks.} Lemma~\ref{lem:ext} is an extension of Theorem 2.3 by Reinhardt~\citet{Reinhardt1961:The-Use-of-Least} to finite models. Reinhardt's theorem requires that for any constant $t$, with measure $0$ we have $\pi_{h_0^*}(P)=t\pi_{h_1}(P)$. This is an important assumption in Reinhardt's proof because 
it assumes away cases with $\ratio(P)=k_\alpha$ so that the most powerful test is deterministic. Unfortunately, this assumption does not hold for finite models and we must deal with randomized tests.
}

\section{\MakeUppercase{UMP Tests for Mallows}}
\label{sec:mallows}
In this section, we present results on UMP non-winner and winner tests for Mallows' model.





\noindent{\bf Non-Winner Tests for Mallows.} The first theorem (Theorem~\ref{thm:winnercvsm}) of this subsection is a warmup, whose main goal is to define a test $f_{\alpha,a,B}$ that is UMP for any simple $H_1$ that consists in a linear order where $a$ is not ranked at the top. The main theorem of this section is Theorem~\ref{thm:winnercvsm}, which characterizes all UMP non-winner tests for arbitrary choices $H_1$.

For any profile $P$, any $B\subset \ma$, and any $a\in (\ma-B)$, we let $w_P(B\succ a)=\sum_{b\in B}w_P(b\succ a)$, that is, the total weights on edges from $B$ to $a$ in $\wmg(P)$.
\begin{thm}\label{thm:winnercvsm} {\bf (A most powerful non-winner test for Mallows)} Given a  Mallows' model $\mallows$, for any alternative $a$, any ranking $h_1$ where $a$ is not ranked at the top, any $0<\alpha<1$, and any $n$, the following test is a level-$\alpha$ UMP for testing $\atopr$ vs.~$h_1$. For any $n$-profile $P_n$,
$$f_{\alpha,a,B}(P_n)=\left\{\begin{array}{ll}1&\text{if }w_{P_n}(B\succ a)>K_\alpha\\
0&\text{if }w_{P_n}(B\succ a)<K_\alpha\\ 
\Gamma_\alpha&\text{if }w_{P_n}(B\succ a)=K_\alpha
\end{array}\right.,$$ where $B$ is the set of alternatives ranked above $a$ in $h_1$, and $K_\alpha$ and $\Gamma_\alpha$ are chosen s.t.~the size of $f_{\alpha,a,B}$ is $\alpha$. 
\end{thm}
\begin{proof}  The proof proceeds by identifying a  uniformly least favorable distribution for $H_0=\atopr$ vs.~$h_1$. In fact, let $B$ denote the set of alternatives ranked above $a$ in $h_1$. Let $h_0^*$ denote the ranking that is obtained from $h_1$ by raising $a$ to the top position. We will prove that the deterministic distribution $\Lambda$ at $\{h_0^*\}$ is a uniformly least favorable distribution.

Let $\lr_\alpha$ denote $\lr_{\alpha,h_0^*,h_1}$ and let $\ratio$ denote $\ratio_{h_0^*,h_1}$. Recall that both are defined in Definition~\ref{dfn:lrtest}. We first prove the theorem for $n=1$. 
By Lemma~\ref{lem:dom}, it suffices to prove that for any $h_0\in H_0$, $X_{h_0^*}^\Lambda$ weakly dominates $X_{h_0}^\Lambda$.
For any ranking $V$ and any pair of alternatives $b,c$, we let $I(b\succ_V c)=1$ if $b\succ_V c$, otherwise $I(b\succ_V c)=0$.  For any single-vote profile $P=\{V\}$, we have:
\begin{align*}
&\log \ratio(P)=(\kt(V,h_1)-\kt(V,h_0^*))\log \varphi \\
=&\log \varphi \sum_{c\succ_V d} (I(d\succ_{h_1} c)-I(d\succ_{h_0^*} c))\\
=&\log \varphi (\sum_{b\in B: a\succ_V b} (I(b\succ_{h_1} a)-I(b\succ_{h_0^*} a))\\
&+\sum_{b\in B: b\succ_V a } (I(a\succ_{h_1} b)-I(a\succ_{h_0^*} b)))\\
=&\log \varphi\cdot (|B|-2w_{P}(B\succ a))
\end{align*}  
Therefore, to prove that $X_{h_0^*}^\Lambda$ weakly dominates $X_{h_0}^\Lambda$, it suffices to prove for any $K\in\mathbb Z$, 

\resizebox{1\linewidth}{!}{$
\hfill\pi_{h_0}(\{P: w_{P}(B\succ a)\geq K\})\leq \pi_{h_0^*}(\{P: w_{P}(B\succ a)\geq K\})\hfill$}

Let $M$ denote the permutation over $\ma$ such that $M(h_0)=h_0^*$. Because $h_0\in H_0=\atopr$, we have $M(a)=a$. Let $B'=M(B)$. Because  Kendall-Tau distance is invariant to permutations, for any $P\in\ml(\ma)$ we have $\pi_{h_0}(P)=\pi_{M(h_0)}(M(P))$ and 
\begin{align*}
&\pi_{h_0}(\{P: w_{P}(B\succ a)\geq K\})\\
=&\pi_{M(h_0)}(\{M(P): w_{M(P)}(M(B)\succ M(a))\geq K\})\\
=&\pi_{h_0^*}(\{M(P): w_{M(P)}(B'\succ M(a))\geq K\})\\
=&\pi_{h_0^*}(\{P: w_{P}(B'\succ a)\geq K\})
\end{align*}
Therefore, it suffices to prove that $
\pi_{h_0^*}(\{P: w_{P}(B'\succ a)\geq K\})\leq \pi_{h_0^*}(\{P: w_{P}(B\succ a)\geq K\})$. 
We will prove a stronger lemma. 
Given any $W\in\ml(\ma)$ and $C', C\subseteq\ma$ with $C\ne C'$ and $|C|=|C'|$, we say that $C$ {\em dominates} $C'$ w.r.t.~$W$ if there exists a one-one mapping $F:(C- C')\ra (C'- C)$ such that for all $c\in C$ we have $c\succ_{W}F(c)$. In words, $C'$ can be obtained from $C$ by lowering some alternatives according to $W$.
 
\begin{lem}\label{lem:CC'} Under a Mallows' model, for any $\varphi$, any $K\in \mathbb N$, any $a\in \ma$, any $W\in\ml(\ma)$, and any $C',C\subseteq\ma$ such that $C$ dominates $C'$ w.r.t.~$W$, we have $\pi_{W}(\{P:w_P(C'\succ a)\geq K\})\leq \pi_{W}(\{P:w_P(C\succ a)\geq K\})$.
\end{lem}

It follows from Lemma~\ref{lem:CC'} that $X_{h_0^*}^\Lambda$ weakly dominates $X_{h_0}^\Lambda$, which means that $\Lambda$ is a uniformly least favorable distribution for $n=1$ by Lemma~\ref{lem:dom}. We note that $\Lambda$ is deterministic. Therefore, by Lemma~\ref{lem:ext}, $\Lambda$ is also a uniformly least favorable distribution for Mallows' model with any $n\in\mathbb N$, which means that the corresponding likelihood ratio test $\lr_\alpha$ is most powerful.  It is not hard to verify that $\lr_\alpha=f_{\alpha,a,B}$.  Moreover, because $\Lambda$ is deterministic, any most powerful test $f$ for $H_0$ vs.~$h_1$ must also be most powerful for the simple vs.~simple test ($h_0^*$ vs.~$h_1$). By the Neyman-Pearson lemma (Lemma~\ref{lem:NP}),  $f$ must agree with $\lr_\alpha$ except on $P_n$ such that $\ratio(P_n)=k_\alpha$, which corresponds to $P_n$ with $w_{P_n}(B\succ a)=K_\alpha$. 
\end{proof}

Theorem~\ref{thm:winnercvsm} can be extended to the following characterization of all UMP non-winner tests ($H_0=\atopr$) for Mallows' model. For any $B\subset \ma$ and $a\in (\ma\setminus B)$, we let $L_{B\succ a}\subseteq \ml(\ma)$ denote the set of all rankings where the set of alternatives ranked above $a$ is exactly $B$. For example, when $m=4$, $L_{\{c\}\succ a}=\{[c\succ a\succ b\succ d], [c\succ a\succ d\succ b]\}$.

\begin{thm} \label{thm:tvcm} {\bf (Characterization of UMP non-winner tests for Mallows)} Given  a Mallows' model $\mallows$  with $m\ge 2$ and $n\ge 2$, there exists a UMP test for $H_0=\atopr$ vs.~$H_1$ for all $0<\alpha<1$ if and only if there exists $B\subseteq \ma$ such that $H_1\subseteq L_{B\succ a}$. 

Moreover, when $H_1\subseteq L_{B\succ a}$, we have that $f_{\alpha,a,B}$ as defined in Theorem~\ref{thm:winnercvsm} is a UMP test.
\end{thm}

\begin{ex}\label{ex:umpmallows}\rm Let $P_7$ denote the profile mentioned in the beginning of Introduction. Suppose we want to test whether there is enough evidence to claim that $a$ cannot be the winner. We can apply a non-winner test on $a$ by letting $H_0= L_{a\succ\text{others}}$ and $H_1= L_{\text{others}\succ a}$. By Theorem~\ref{thm:tvcm}, $f_{\alpha,a,B}$ is a UMP test, where $B = \{b,c\}$. The test can be done by computing the test statistic ${\calT} = w_{P_7}(B\succ a) = -2$, and then checking if $\calT$ is in the critical region $(K_\alpha,\infty)$ for some pre-computed $K_\alpha$. If $\calT\in (K_\alpha,\infty)$, then $H_0$ is rejected, which means that $a$ should not be chosen as the winner. If $\calT = K_\alpha$, then $H_0$ is rejected with probability $\Gamma_\alpha$. Otherwise $H_0$ cannot be rejected, meaning that there is not enough evidence to claim that $a$ cannot be the winner.
 \hfill$\blacksquare$
\end{ex}

\noindent{\bf Winner Tests for Mallows.} We now consider UMP winner tests under Mallows' model ($H_1=\atopr$) for two natural choices of $H_0$: $H_0=L_{\others\succ a}$ in Theorem~\ref{thm:winnerbottomtopm}, which means that $a$ is ranked in the bottom in the ground truth, and $H_0=(\ml(\ma)- H_1)$ in Theorem~\ref{thm:mallowswinnernonex} and~\ref{thm:mallowswinnerexist}, which means that $a$ is not ranked at the top in the ground truth. 

\begin{thm}[\bf A UMP winner test under Mallows]\label{thm:winnerbottomtopm} Given a  Mallows' model $\mallows$, for any alternative $a$, any $0<\alpha<1$, and any $n$, the following test is a level-$\alpha$ UMP for testing $H_0=L_{\others\succ a}$ vs.~$H_1=\atopr$. For any $n$-profile $P_n$,
$$f_{\alpha,a}(P_n)=\left\{\begin{array}{cc}1&\text{if }w_{P_n}( a\succ \others)>K_\alpha\\
0&\text{if }w_{P_n}(a\succ\others)<K_\alpha\\
\Gamma_\alpha&\text{if }w_{P_n}(a\succ \others)=K_\alpha
\end{array}\right.,$$
where $K_\alpha$ and $ \Gamma_\alpha$ are chosen s.t.~the size of $f_{\alpha,a}$ is $\alpha$.
\end{thm}
\begin{proof} 
For any $h_1\in H_1$, we will prove that $f_{\alpha,a}$ is a most powerful level-$\alpha$ test. Let $h_0^*\in H_0$ denote the ranking that is obtained from $h_1$ by moving $a$ to the bottom position without changing the relative positions of the other alternatives. Like the proof of Theorem~\ref{thm:winnercvsm}, it is not hard to check that $f_{\alpha,a}$ is equivalent to the likelihood ratio test $\lr_{\alpha,h_0^*,h_1}$. 

Because $f_{\alpha,a}$ is invariant to permutations over $\ma\setminus\{a\}$, for any $h_0'\in H_0$ and any permutation $M$ over $\ma\setminus\{a\}$, we have $\size(f_{\alpha,a},h_0')=\size(f_{\alpha,a},M(h_0'))$. In particular, let $M$ denote the permutation such that $M(h_0')=h_0^*$. We have $\size(f_{\alpha,a},h_0')=\size(f_{\alpha,a},h_0^*)$. It follows from Lemma~\ref{lem:hvs} that $f_{\alpha,a}$ is most powerful, by letting $\Lambda$ to be the deterministic distribution on $\{h_0^*\}$. 
\end{proof}

\begin{ex}\label{ex:umpmallowswt}\rm Let us continue with the setting in Example~\ref{ex:umpmallows}. Suppose we want to test whether there is enough evidence to claim that $a$ is the winner. We can apply a winner test on $a$ by letting $H_0= L_{\text{others}\succ a}$ and $H_1= L_{a\succ\text{others}}$, i.e.~switching the roles of $H_0$ and $H_1$ in Example~\ref{ex:umpmallows}. By Theorem~\ref{thm:winnerbottomtopm}, $f_{\alpha,a}$ is a UMP test. The test can be done by computing the test statistic ${\calT} = w_{P_7}(a\succ\text{others}) = 2$, and then checking if $\calT$ is in the critical region $(K_\alpha^*,\infty)$ for some pre-computed $K_\alpha^*$. If $\calT\in (K_\alpha^*,\infty)$, then $H_0$ is rejected, which means that $a$ should be chosen as the winner. If $\calT = K_\alpha^*$, then $H_0$ is rejected with a pre-computed probability $\Gamma_\alpha^*$. Otherwise $H_0$ cannot be rejected, meaning that there is no enough evidence to claim that $a$ is the winner.
 \hfill$\blacksquare$
\end{ex}

The following two theorems identify conditions on $\varphi$ in Mallows' model for the UMP winner test  $H_0=(\ml(\ma)\setminus H_1)$ vs.~$H_1=\atopr$ when $n=1$.

\begin{thm}\label{thm:mallowswinnernonex} Let $\mallows$ denote a  Mallows' model with $n=1$, any $m\ge 4$, and any $\varphi< 1/m$. There exists $0<\alpha<1$ such that no level-$\alpha$ UMP test exists for $H_0=(\ml(\ma)- H_1)$ vs.~$H_1=\atopr$.
\end{thm}

\begin{thm}\label{thm:mallowswinnerexist} Let $\mallows$ denote a Mallows' model with $n=1$ and any $m\ge 4$. There exists  $\epsilon>0$ such that for any $\varphi> 1-\epsilon$ and any $\alpha$, a UMP test exists for  $H_0=(\ml(\ma)- H_1)$ vs.~$H_1=\atopr$.
\end{thm}
%


\section{\MakeUppercase{UMP Tests for Condorcet}}
\label{sec:condorcet}
We first prove two general theorems on UMP tests for statistical models that combine multiple independent models, and then apply them to characterize UMP tests under Condorcet's model. 

\begin{dfn}{\bf (Combining two models)}
Given two models $\mm_X=(\ms_X,\Theta_X,\vec \pi_X)$ and $\mm_Y=(\ms_Y,\Theta_Y,\vec \pi_Y)$, we let $\mm_{X}\otimes \mm_Y=(\ms_X\times \ms_Y,\Theta_X\times\Theta_Y,\vec\pi_{X}\times \vec\pi_Y)$, where  for any $(\pi_{\theta_X},\pi_{\theta_Y})\in \vec \pi_X\times \vec \pi_{Y}$ and any $P_X\in\ms_X$ and $P_Y\in\ms_Y$, we let $(\pi_{\theta_X},\pi_{\theta_Y})(P_X,P_Y)=\pi_{\theta_X}(P_X)\cdot \pi_{\theta_Y}(P_Y)$.
\end{dfn}

\begin{ex}\label{ex:condorcetm} Given a Condorcet's model $\condorcet$ with $m=3$. Let $\ma=\{1,2,3\}$. For any pair of alternatives $\{a,b\}$, we let $\mm_{\{a,b\}}=(\{0,1\}^n,\{0,1\},\vec\pi)$ denote the restriction of $\condorcet$ on the pairwise comparison between $a$ and $b$. We have $\condorcet=\mm_{\{1,2\}}\otimes \mm_{\{2,3\}}\otimes \mm_{\{1,3\}}$. \hfill$\blacksquare$
\end{ex}

Given two models $\mm_X$ and $\mm_Y$, the next theorem provides a way to leverage a least favorable distribution for a composite vs.~simple test under $\mm_X$ to a least favorable distribution for a composite vs.~simple test under the combined model $\mm_X \otimes \mm_Y$.
\begin{lem}\label{lem:ind}
For any pair of models $\mm_X$ and $\mm_Y$, suppose $\Lambda_X$ is a least favorable distribution for composite vs.~simple test ($H_{0,X}$ vs.~$x_1$) under $\mm_X$. For any $y_1\in \Theta_Y$, let $\Lambda^*$ be the distribution over $H_{0,X}\times\Theta_Y$ where for all $x\in H_{0,X}$, $\Lambda^*(x,y_1)=\Lambda_X(x)$. Then, $\Lambda^*$ is a least favorable distribution for $H_{0,X}\times \Theta_Y$ vs.~$(x_1,y_1)$ under $\mm_X\otimes \mm_Y$.
\end{lem}

\begin{ex} \rm Continuing Example~\ref{ex:condorcetm}, we let $\mm_X=\mm_{\{1,2\}}$, $H_{0,X}=\{0\}$, $x_1=1$, let $\Lambda_X$ be the deterministic distribution over $\{0\}$, and let $\mm_Y=\mm_{\{2,3\}\times \{1,3\}}$ and $y_1=(1,1)$. $\Lambda_X$ is a least favorable distribution according to the Neyman-Pearson lemma (Lemma~\ref{lem:NP}). Let $\Lambda^*$ denote the deterministic distribution over $\{(0,1,1)\}$. It follows from Lemma~\ref{lem:ind} that $\Lambda^*$ is a least favorable distribution for $(\{0\}\times \{0,1\}^2)$ vs.~$(1,1,1)$ under Condorcet's model. \hfill$\blacksquare$
\end{ex}

The next theorem focuses on the setting where we combine $t\in \mathbb N$  identical statistical models $\mm_X$. Given $\mm_X=(\ms,\Theta,\vec\pi)$, a distribution $\Lambda$ over $\Theta$, any $\theta^*\in \Theta$, and any $t\in \mathbb N$, we let $(\mm_X)^t=\underbrace{\mm_X\otimes\cdots\otimes\mm_X}_t$ and define the extension of $\Lambda$ to $\Theta^t$ w.r.t.~$\theta^*$, denoted by $\ext(\Lambda,\theta^*,t)$, as follows. Let $\vec \theta^*=(\theta^*,\ldots,\theta^*)\in \Theta^t$. For any $j\in t$ and any $\theta\in \Theta$, we have $\ext(\Lambda,\theta^*,t)(\theta,[\vec \theta^*]_{-j})=\frac{1}{t}\Lambda(\theta)$. That is, $\ext(\Lambda,\theta^*,t)$ generates a vector $\vec \theta\in \Theta^t$ in the following two steps. First, a number $j\le t$ is chosen uniformly at random. Then, we fix the components of $\vec \theta$ to be $\theta^*$, except for the $j$-th component, which is generated from $\Theta$ according to $\Lambda$.  

For any $H_0\subseteq \Theta$ and any $h_1\in (\Theta\setminus H_0)$, we let $\vec h_1=(\underbrace{h_1,\ldots,h_1}_t)$ and let $\ext(H_0,h_1,t)=(\{H_0\cup \{h_1\}\}^t\setminus \{\vec h_1\})$.  

\begin{ex}\label{ex:combext}\rm In the setting of Example~\ref{ex:condorcetm}, we let $\mm_X=\mm_{\{1,2\}}$, let $\Lambda$ denote the deterministic distribution over $\{0\}$, let $H_0=\{0\}$ and $h_1=1$. Then, $\ext(\Lambda,1,3)$ is the uniform distribution over $\{(0,1,1), (1,0,1),(1,1,0)\}$, $\vec h_1=(1,1,1)$, and $\ext(H_0,1,3)=(\{0,1\}^3\setminus\{(1,1,1)\})$. \hfill$\blacksquare$
\end{ex}

\begin{lem}\label{lem:extdomain}For any model $\mm_X$ and any $t\in\mathbb N$, suppose $\Lambda$ is a uniformly least favorable distribution for composite vs.~simple test ($H_0$ vs.~$h_1$) under $\mm_X$. Then $\ext(\Lambda,h_1,t)$ is a uniformly least favorable distribution for $\ext(H_0,h_1,t)$ vs.~$\vec h_1$ in $(\mm_X)^t$.
\end{lem}

\begin{ex}\rm In the setting of Example~\ref{ex:combext}, it follows from Lemma~\ref{lem:extdomain} that 
the uniform distribution over $\{(0,1,1), (1,0,1),(1,1,0)\}$ is a uniformly least favorable distribution for testing $\ext(H_0,1,3)=(\{0,1\}^3\setminus\{\vec 1\})$ vs.~$\vec h_1=(1,1,1)$ under $(\mm_{X})^3$, which is the Condorcet's model with $m=3$. \hfill$\blacksquare$
\end{ex}

\noindent{\bf Non-Winner Tests for Condorcet.} We are now ready to  characterize UMP tests for Condorcet's model by applying Lemma~\ref{lem:ind} and~\ref{lem:extdomain}. Theorem~\ref{thm:winnercvsc}  and Theorem~\ref{thm:tvcc} of this section are counterparts of Theorem~\ref{thm:winnercvsm} and Theorem~\ref{thm:tvcm} (both are for Mallows' model), respectively, though the proof techniques are quite different.
\begin{thm}\label{thm:winnercvsc}  {\bf (A most powerful non-winner test for Condorcet)}
Given a Condorcet's model $\condorcet$ with $m\ge 2$, for any $a\in\ma$, any $h_1\in (\mb(\ma)\setminus\atopb)$, any $n$, and any $\varphi$, the following test is most powerful for testing $\atopb$ vs.~$h_1$. 
For any $n$-profile $P_n$,
$$g_{\alpha,a,B}(P_n)=\left\{\begin{array}{cc}1&\text{if }w_{P_n}(B\succ a)>K_\alpha\\
0&\text{if }w_{P_n}(B\succ a)<K_\alpha\\
\Gamma_\alpha&\text{if }w_{P_n}(B\succ a)=K_\alpha
\end{array}\right.,$$
where $B$ is the set of alternatives that are preferred to $a$ in $h_1$. 
\end{thm}
\begin{proof} 
Let $h_0^*$ denote the binary relation obtained from $h_1$ by enforcing $a\succ b$ for all $b\in \ma$. We will prove that the deterministic distribution over $\{h_0^*\}$ is a uniformly least favorable distribution for $\atopb$ vs.~$h_1$. 

Let $X=\{\{a,b\}:b\ne a\}$ denote the pairwise comparisons between alternatives in $\ma$ that involve $a$ and let $Y$ denote the set of all other pairwise comparisons. Let $\mm_X=(\ms_X,\Theta_X,\vec \pi_X)$ denote Condorcet's model $\condorcet$ restricted to $X$. That is, $\ms_X= \{0,1\}^{(m-1)n}, \Theta_X=\{0,1\}^m$ and  for any $\theta\in \Theta_X$ and any $P_n\in \ms_X $,  $\pi_\theta(P_n)\propto \varphi^{\kt(\theta, P_n)}$. Similarly, let $\mm_Y$ denote Condorcet's model restricted to $Y$.  It follows that $\condorcet=\mm_X\otimes \mm_Y$.

Let $h_1=(x_1,y_1)$, where $x_1\in \Theta_X$ and $y_1\in \Theta_Y$. Let $x_0\in\Theta_X$ denote the vector that represents $a\succ b$ for all $b\in\ma$. By Neyman-Pearson lemma (Lemma~\ref{lem:NP}), the deterministic distribution $\Lambda_X=\{x_0\}$ is a uniformly least favorable distribution for $x_0$ vs.~$x_1$. Therefore, by Lemma~\ref{lem:ind}, the deterministic distribution $\Lambda=\{(x_0,y_1)\}$ is  uniformly least favorable for $\{x_0\}\times \Theta_Y$ vs.~$(x_1,y_1)$. We note that $(x_0,y_1)=h_0^*$ and $(x_1,y_1)=h_1$. It is not hard to verify that $g_{\alpha,a,B}$ is equivalent to the likelihood ratio test $\lr_{\alpha,\Lambda,h_1}$, which is most powerful. The theorem follows after Lemma~\ref{lem:hvs}.\end{proof}

Subsequently, we have the following characterization of UMP non-winner tests under Condorcet's model ($H_0=\atopb$).
For any $B\subset \ma$, we let $R_{B\succ a}\subseteq \mb(\ma)$ denote the set of all binary relations where the set of alternatives that are preferred to $a$ is $B$.

\begin{thm}\label{thm:tvcc} {\bf (Characterization of UMP non-winner tests for Condorcet)}
Let $\condorcet$ denote a Condorcet's model with any $m\ge 2$ and $n\ge 2$. There exists a UMP test for $H_0=\atopb$ vs.~$H_1$ for every $0<\alpha<1$ if and only if there exists $B\subseteq \ma$ such that $H_1\subseteq R_{B\succ a}$. 

Moreover, when $H_1\subseteq R_{B\succ a}$, $g_{\alpha,a,B}$ defined in Theorem~\ref{thm:winnercvsc} is a UMP test.
\end{thm}
The proof is similar to the proof of Theorem~\ref{thm:tvcm} and is thus omitted.

\noindent{\bf Winner Tests for Condorcet.} Finally, we turn to UMP winner tests for Condorcet's model ($H_1=\atopb$). 


\begin{thm}[\bf A UMP winner test for Condorcet]\label{thm:winnerotherstopc} Let $\condorcet$ denote a Condorcet's model with any $m\ge 2$, any $n\ge 2$, and any $\varphi$.
For any $\alpha$, $g_{\alpha,a}$ defined below is a level-$\alpha$ UMP test for $H_0=(\mb(\ma)\setminus H_1)$ vs.~$H_1=\atopb$.  For any $P_n$,  
$$g_{\alpha,a}(P_n)=\left\{\begin{array}{cc}1&\text{if }\ratio(P_n)>K_\alpha\\
0&\text{if }\ratio(P_n)<K_\alpha\\
\Gamma_\alpha&\text{if }\ratio(P_n)=K_\alpha
\end{array}\right.,$$
where $\ratio(P_n)=\dfrac{m-1}{\sum_{b\ne a}\varphi^{w_{P_n}(a\succ b)}}$, and $K_\alpha$ and $\Gamma_\alpha$ are chosen such that the level of $g_{\alpha,a}$ is $\alpha$.
\end{thm}
\begin{proof} Let $\mm_1$ denote Condorcet's model with a single sample. Let $X_1,\ldots,X_{m-1}$ denote the $m-1$ pairwise comparisons between $a$ and other alternatives. Similarly to the proof of Theorem~\ref{thm:winnercvsc}, we let $\mm_{X_1},\ldots,\mm_{X_{m-1}}$ denote the restriction of $\mm_1$ on the $m-1$ pairwise comparisons, and let $\mm_Y$ denote the restriction of $\condorcet$ on other pairwise comparisons. In fact, $\mm_{X_1},\ldots\mm_{X_{m-1}}$ are the same model. It follows that $\condorcet=\mm_{X_1}\otimes\mm_{X_2}\otimes\cdots \otimes\mm_{X_{m-1}}\otimes\mm_Y$.

In $\mm_{X_1}$, let $1$ represent that $a$ is more preferred in the pairwise comparison. Due to the Neyman-Pearson lemma (Lemma~\ref{lem:NP}), the deterministic distribution $\Lambda=\{0\}$ is a uniformly least favorable distribution for $H_0=\{0\}$ vs.~$h_1=1$. For any $n\in\mathbb N$, let $\mm_{X_{1,n}}$ denote $\mm_{X_1}$ with $n$ i.i.d.~samples. It follows from Lemma~\ref{lem:ext} that $\Lambda$ is still a uniformly least favorable distribution for $\mm_{X_{1,n}}$. By Lemma~\ref{lem:extdomain}, $\ext(\Lambda,h_1,m-1)$ is a uniformly least favorable distribution for $\ext(H_0,h_1,m-1)=(\{0,1\}^{m-1}\setminus\{\vec 1\})$ vs.~$h_1=\vec 1$ under $\mm_{X_{1,n}}\otimes\cdots\otimes \mm_{X_{m-1,n}}$. 

Let $\mm_{Y,n}$ denote the model obtained from $\mm_Y$ by using $n$ i.i.d.~samples. For any $y_1\in\Theta_{Y,n}$, let $\Lambda_{y_1}$ denote the distribution that is obtained from $\ext(\Lambda,h_1,m-1)$ by appending $y_1$ to each parameter. By Lemma~\ref{lem:ext}, $\Lambda_{y_1}$  is a uniformly least favorable distribution for $\ext(H_0,h_1,m-1)\times \Theta_{Y,n}$ vs.~$(\vec 1,y_1)$ under $\mm_{X_{1,n}}\otimes\cdots\otimes \mm_{X_{m-1,n}}\otimes \mm_{Y,n}$, which is the Condorcet's model with $n$ i.i.d.~samples. We note that $\ext(H_0,h_1,m-1)\times \Theta_{Y,n}=(\{0,1\}^{m-1}\setminus\{\vec 1\})\times \Theta_Y=(\mb(\ma)\setminus\atopb$). This means that the likelihood ratio test $\lr_{\alpha,\Lambda_{y_1},(\vec 1,y_1)}$ is a most powerful level-$\alpha$ test for $(\mb(\ma)\setminus\atopb)$ vs.~$(\vec 1,y_1)$. We note that for all $y_1$, $\lr_{\alpha,\Lambda_{y_1},(\vec 1,y_1)}$ is the same test, which means that it is also UMP. The theorem is proved after noticing that $g_{\alpha,a}=\lr_{\alpha,\Lambda_{y_1},(\vec 1,y_1)}$.
\end{proof}


\section{\MakeUppercase{Discussion: Beyond Binary Choice}}\label{sec:discussion}
%
%
All UMP tests we have characterized so far are optimal in making binary decisions, such as whether a given alternative $a$ is the winner. We propose two natural procedures to choose the winner by combining multiple winner tests ($H_1=\atopr$ for Mallows' model and $H_1=\atopb$ for Condorcet' model) and non-winner tests ($H_0=\atopr$ for Mallows' model and $H_0=\atopb$ for Condorcet' model).

\noindent{\bf Procedure based on combining winner tests.} We first choose any winner test, such as a UMP test characterized in Theorem~\ref{thm:winnerbottomtopm}, then find the alternative $a$ with the minimum $\alpha$ such that $H_0$ is rejected in the winner test, by conducting binary search on $\alpha$.\footnote{Co-winners exist if they all reject $H_0$ for the same $\alpha$.} This corresponds at a high level to choosing the alternative that is most likely to be the winner according to the tests. 

\noindent{\bf Procedure based on combining non-winner tests.} Similarly, we use binary search on $\alpha$ to find the alternative $a$ with the maximum $\alpha$ such that $H_0$ is rejected in the non-winner test. This corresponds to choosing the alternative that is mostly unlikely to be a non-winner according to the tests.


Interestingly, both procedures correspond to the Borda voting rule when the proposed UMP tests for Mallows' model are used: in  the UMP winner test we let $H_{0}=L_{\others\succ a}$ vs.~$H_1=L_{a\succ\others}$ as in Example~\ref{ex:umpmallowswt}, and in  the UMP non-winner test we let $H_0=L_{a\succ\others}$ vs.~$H_{1}=L_{\others\succ a}$ as in Example~\ref{ex:umpmallows}. This provides a new theoretical justification for the Borda rule; or vice versa, Borda provides a justification of the proposed procedure.

\section{\MakeUppercase{Future Work}}

An immediate open question is  how to use hypothesis testing for choosing a winner beyond testing whether a given alternative is a winner or not, following  the initial thoughts discussed in Section~\ref{sec:discussion}. 
Also, can we characterize UMP tests for other goals of social choice, such as pairwise comparisons? Do UMP tests exist for other statistical models, such as  random utility models?  How can we efficiently compute the results of the proposed tests? 


\onecolumn
\newpage
\section{Appendix: Proofs}

\vspace{3mm}
\noindent{\bf Lemma~\ref{lem:ext}.} {\em Suppose $\Lambda$ is a deterministic uniformly least favorable distribution for composite vs.~simple test ($H_0$ vs.~$h_1$) under $\mm=(\ms,\Theta,\vec\pi)$. Then for any $n\in \mathbb N$, $\Lambda$ is also a uniformly least favorable distribution for testing $H_0$ vs.~$h_1$ under $\mm=(\ms^n,\Theta,\vec\pi)$ with $n$ i.i.d.~samples.}

\begin{proof}  Let $\supp(\Lambda)=\{h_0^*\}$.  For any $n\in\mathbb N$ and any $h_0\in H_0$, we define a random variable $X_{n,h_0}:\ms^n\ra \mathbb R$, where for any $P_n\in\ms^n$, $\Pr(P_n)=\pi_{h_0}(P_n)=\prod_{V\in P_n}\pi_{h_0}(V)$, and $X_{n,h_0}(P_n)=\log \ratio_{h_0^*,h_1}$. It follows that 
$$X_{n,h_0}=\underbrace{X_{h_0}+ X_{h_0}+ \cdots+ X_{h_0}}_n$$ 

By Lemma~\ref{lem:dom},  for any $h_0\in H_0$, $X_{h_0^*}$ weakly dominates $X_{h_0}$. Because first-order stochastic dominance is preserved under convolution~\citep{Deelstra14:Risk}, we have that $X_{n,h_0^*}$ weakly dominates $X_{n,h_0}$. The lemma follows after applying Lemma~\ref{lem:dom}. \end{proof}

{
\vspace{2mm}
\noindent{\bf Remarks.} Lemma~\ref{lem:ext} is an extension of Theorem 2.3 by Reinhardt~\citet{Reinhardt1961:The-Use-of-Least} to finite models. Reinhardt's theorem requires that for any constant $t$, with measure $0$ we have $\pi_{h_0^*}(P)=t\pi_{h_1}(P)$. This is an important assumption in Reinhardt's proof because 
it assumes away cases with $\ratio(P)=k_\alpha$ so that the most powerful test is deterministic. Unfortunately, this assumption does not hold for finite models and we must deal with randomized tests.
}

\vspace{3mm}
\noindent{\bf Lemma~\ref{lem:CC'}} {\em  Under a Mallows' model, for any $\varphi$, any $K\in \mathbb N$, any $a\in \ma$, any $W\in\ml(\ma)$, and any $C',C\subseteq\ma$ such that $C$ dominates $C'$ w.r.t.~$W$, we have $\pi_{W}(\{P:w_P(C'\succ a)\geq K\})\leq \pi_{W}(\{P:w_P(C\succ a)\geq K\})$.}

\begin{proof} We first prove the lemma for a special case where $C$ and $C'$ differ in only one alternative, that is, $|C- C'|=1$. Let $c\in C$ such that $c\not\in C'$. Let $c'\in C'$ such that $c'\not\in C$. Because $C$ dominates $C'$ in $W$, we have $c\succ_{W}c'$. 

Let $\mcp=\{P\in\ml(\ma):w_P(C\succ a)\geq K\}$ and  $\mcp'=\{P\in\ml(\ma):w_P(C'\succ a)\geq K\}$. We define the following permutation $\mm$ over $\ml(\ma)$. For any $P\in \ml(\ma)$, if $c\succ_P a \succ_P c'$ then $\mm(P)$ is the ranking that is obtained from $P$ by switching $c$ and $c'$; otherwise $\mm(P)=P$. Because $|C- C'|=1$, it follows that for any $P\in \mcp- \mcp'$, we must have $c\succ_P a\succ_P c'$ and $(C- C')\succ_P a$. Therefore, $\mm(\mcp- \mcp')=\mcp'- \mcp$.

We now prove that $\pi_{W}(\mcp-\mcp')>\pi_{W}(\mcp'-\mcp)$.  For any $P\in\mcp-\mcp'$, we have  $c\succ_P a\succ_P c'$, which means that $\pi_{W}(P)\geq \pi_{W}(\mm(P))/\varphi$ because $c\succ _{W} c'$. Therefore, $\pi_{W}(\mcp-\mcp')>\pi_{W}(\mcp'-\mcp)$ because $\mm(\mcp-\mcp')=\mcp'-\mcp$. 

We have 
$
\pi_{W}(\mcp)
=\pi_{W}(\mcp\cap\mcp')+\pi_{W}(\mcp-\mcp')
\ge \pi_{W}(\mcp\cap\mcp')+\pi_{W}(\mcp'-\mcp)
=\pi_{W}(\mcp')
$.

Therefore, the lemma holds for the case where $|C- C'|=1$. For general $C$ and $C'$, because $C$ dominates $C'$, there exists a sequence of sets $C=C_0, C_1,\ldots, C_l=C'$ such that  for all $0\leq i\leq l-1$, (i) $C_i$ dominates $C_{i+1}$; (ii) $|C_i- C_{i+1}|=1$. It follows that 
$\pi_{W}(\{P:w_P(C\succ a)\geq K\})\ge \pi_{W}(\{P:w_P(C_1\succ a)\geq K\})\ge \cdots\ge \pi_{W}(\{P:w_P(C'\succ a)\geq K\})$.
\end{proof}

\vspace{5mm}\noindent
{\bf Theorem~\ref{thm:tvcm} (Characterization of all UMP non-winner tests under Mallows).} {\em Given  a Mallows' model $\mallows$  with $m\ge 2$ and $n\ge 2$, there exists a UMP test for $H_0=\atopr$ vs.~$H_1$ for all $0<\alpha<1$ if and only if there exists $B\subseteq \ma$ such that $H_1\subseteq L_{B\succ a}$. 

Moreover, when $H_1\subseteq L_{B\succ a}$, $f_{\alpha,a,B}$ defined in Theorem~\ref{thm:winnercvsm} is a UMP test.}

\begin{proof} {The ``if'' part.} We note that $f_{\alpha,a,B}$ does not depend on the orderings among alternatives in $B$ in $h_1$. It follows that for all $h_1\in H_1$, $f_{\alpha,a,B}$ is a level-$\alpha$ most powerful test for $H_0$ vs.~$\{h_1\}$, which means that $f_{\alpha,a,B}$ is a UMP test.

{The ``only if'' part.} Suppose there exist $B, B'$ such that $B\ne B'$ and there exist two rankings $h_1^1=[B\succ a\succ \text{others}]$ and $h_1^2=[B'\succ a\succ \text{others}]$ in $H_1$. W.l.o.g.~suppose $B'- B\ne\emptyset$. Let $\alpha$ denote the number such that $K_\alpha=n|B|-0.5$, $\Gamma_\alpha=0$, and let $f_{\alpha,a,B}$ denote the most powerful test for $H_0$ vs.~$h_1^1$ guaranteed by Theorem~\ref{thm:winnercvsm}. Because $K_\alpha$ is not an integer, there does not exist $P_n$ such that $w_{P_n}(B\succ a)=K_\alpha$. This means that $f_{\alpha,a,B}$ is the unique most powerful level-$\alpha$ test for $H_0$ vs.~$h_1^1$. We observe that for any $P_n$, $f_{\alpha,a,B}(P_n)$ is either $0$ or $1$, and $f_{\alpha,a,B}(P_n)=1$ if and only if $a$ is ranked below $B$ in all $n$ rankings in $P_n$. It follows that $f_{\alpha,a,B}$ must be the unique level-$\alpha$ UMP test for $H_0$ vs.~$H_1$. 

By Theorem~\ref{thm:winnercvsm}, any most powerful level-$\alpha$ test, in particular $f_{\alpha,a,B}$, must agree with $f_{\alpha,a,B'}$ except for the threshold cases $w_{P_n}(B'\succ a)=K_\alpha'$ for some $K_\alpha'$. Choose arbitrary $b'\in B'- B$ and $b\in B$.  Let $P_n^*$ be composed of $n$ copies of $[B\succ a\succ\others]$ and let $P_n'$ be composed of $n-1$ copies of $[b'\succ B\succ a\succ\others]$ and one copy of $[b'\succ (B-\{b\})\succ a\succ \others]$. Because $w_{P_n^*}(B\succ a)=n|B|>K_\alpha$, we have $f_{\alpha,a,B}(P_n^*)=1$. This means that the threshold $K_\alpha'$ for $f_{\alpha,a,B'}$ is no more than $w_{P_n^*}(B'\succ a)=n|B\cap B'|$. Because 
$n\ge 2$, we have $w_{P_n'}(B'\succ a)\ge  n(|B\cap B'|+1)-1>n|B\cap B'|=w_{P_n'}(B'\succ a)$, which means that $f_{\alpha,a,B}(P_n')=1$. However, $w_{P_n'}(B\succ a)=n|B|-1<n|B|$, which is a contradiction because for any profile $P_n$, $f_{\alpha,a,B}(P_n)=1$ if and only if $B\succ a$ in all $n$ rankings in $P_n$. 
\end{proof}

{\vspace{5mm}\noindent \bf Theorem~\ref{thm:mallowswinnernonex}.} Let $\mallows$ denote a  Mallows' model with $n=1$, any $m\ge 4$, and any $\varphi< 1/m$. There exists $0<\alpha<1$ such that no level-$\alpha$ UMP test exists for $H_0=(\ml(\ma)- H_1)$ vs.~$H_1=\atopr$.

\begin{proof} By Lemma~\ref{lem:bordatest}, if a UMP test exists then $\bar f_{\alpha,a}$ is also a UMP test. Therefore, it suffices to prove that $\bar f_{\alpha,a}$ is not a level-$\alpha$ UMP test. To this end, we explicitly construct a test $f$ and prove that the rankings assigned value $1$ are more cost-effective than that under $\bar f_{\alpha,a}$.


Let $V_1,V_2,\ldots,V_m, V_2'\in\ml(\ma)$ denote $m+1$ rankings defined as follows. For any $j\le m$, let $V_j=[a_j\succ\others]$, where alternatives in ``$\others$'' are ranked w.r.t.~the increasing order of their subscripts. In other words, $V_j$ is obtained from $V_1$ by raising alternative $a_j$ to the top position. We let $V_3'=[a_3\succ a_1\succ a_4\succ a_2\succ \others]$.

We consider the following critical function $f$. For any $V\in\ml_{a\succ\others}$, we let $f(V)=1$. For any $V_j$ with $j\ne 3$, let $f(V_j)=1$. We then let $f(V_3)=f(V_3')=\frac{1+\varphi^{m}}{1+\varphi}$. Let $\alpha$ denote the size of $f$ at $V_2$. That is, $\alpha=\size(f,V_2)$. Let $T=\pi_{V_2}(\ml_{a\succ\others})$. It follows that 
\begin{align*}
&\alpha-T\\
\propto &\varphi^0+\frac{1+\varphi^m}{1+\varphi}(\varphi^{\kt(V_2,V_3)}+\varphi^{\kt(V_2,V_3')})+\sum_{j=5}^m\varphi^{\kt(V_2,V_j)}\\
=&1+\frac{1+\varphi^m}{1+\varphi}(\varphi^3+\varphi^4)+\varphi^4+ \sum_{j=5}^m\varphi^{\kt(V_2,V_j)}\\
>&1+\varphi^3+\varphi^4 + \varphi^5
\end{align*}

\begin{figure}[htp]
\centering
\includegraphics[trim=0cm 14cm 17cm 0cm, clip=true, width=.26\textwidth]{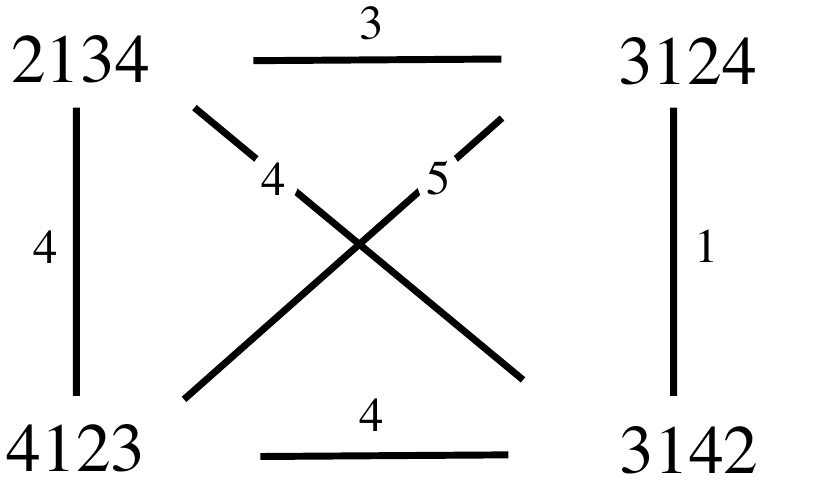}
\caption{Kentall-Tau distance for some rankings over four alternatives.\label{fig:kt4}} 
\end{figure}

For any $j,j^*\ge 2$ such that $j\ne j^*$, it is not hard to verify that $\kt(V_j,V_{j^*})=j+j^*-2$. Moreover, $\kt(V_3,V_3')=1$, $\kt(V_2,V_3')=4$, $\kt(V_4,V_3')=4$, and for any $j\ge 5$, we have $\kt(V_3',V_j)=j+2$. Therefore, we have the following calculations of $\size(f,V_3)$, $\size(f,V_3')$, and $\size(f,V_4)$ (see Figure~\ref{fig:kt4} for distances between $V_2,V_3,V_3',V_4$). We note that $T=\pi_{V_2}(\ml_{a\succ\others})=\pi_{V_3}(\ml_{a\succ\others})=\pi_{V_3'}(\ml_{a\succ\others})=\pi_{V_4}(\ml_{a\succ\others})$ due to symmetry.

$\hfill
\size(f, {V_3})-T\propto\varphi^3+\frac{1+\varphi^m}{1+\varphi}(1+\varphi)+\varphi^5+\sum_{j=5}\varphi^{\kt(V_3,V_j)}\le 1+\varphi^3+(m-3)\varphi^5 \hfill
$

$\hfill
\size(f, {V_3'})-T\propto\varphi^4+\frac{1+\varphi^m}{1+\varphi}(1+\varphi)+\varphi^4+\sum_{j=5}\varphi^{\kt(V_3',V_j)}\le 1+2\varphi^4+ (m-4)\varphi^6\hfill
$

$\hfill
\size(f, {V_4})-T\propto\varphi^4+\frac{1+\varphi^m}{1+\varphi}(\varphi^4+\varphi^5)+1+\sum_{j=5}\varphi^{\kt(V_4,V_j)}\le 1+2\varphi^4+ (m-4)\varphi^7\hfill
$

For any other $h_0'\in H_0$, we have $\size(f,h_0')-T\le m \varphi$. Because $\varphi<1/m$, we have $\size(f)=\alpha$. Let $P$ denote a profile that is composed of $\{V_2,V_4,\ldots,V_m\}\cup \frac{1+\varphi^m}{1+\varphi}\{V_3,V_3'\}$. We next prove that $\ratio_{V_2,V_1}(P)>\ratio_{V_2,V_1}(T_{m-2})$. Let $Z_{m}=\prod_{l=1}^m\frac{1-\varphi^m}{1-\varphi}$ denote the Mallows normalization factor for $m$ alternatives. We have
\begin{align*}
&\ratio_{V_2,V_1}(T_{m-2})=\frac{\pi_{V_1}(T_{m-2})}{\pi_{V_2}(T_{m-2})}\\
=&\frac{\varphi Z_{m-1}}{Z_{m-2}+\varphi^2(Z_{m-1}-Z_{m-2})}\\
=&\frac{\varphi \frac{Z_{m-1}}{Z_{m-2}}}{1+\varphi^2(\frac{Z_{m-1}}{Z_{m-2}}-1)}
=\frac{\varphi+\varphi^2+\cdots+\varphi^{m-1}}{1+\varphi^3+\varphi^4+\cdots+\varphi^m}<\frac{1}{\varphi}
\end{align*}
\begin{align*}
&\ratio_{V_2,V_1}(P)=\frac{\varphi+\varphi^2+\cdots+\varphi^{m-1}+\varphi^{m+2}}{1+\varphi^3+\varphi^4+\cdots+\varphi^m+\varphi^{m+3}}\\
>&\frac{\varphi+\varphi^2+\cdots+\varphi^{m-1}}{1+\varphi^3+\varphi^4+\cdots+\varphi^m}\\
=&\ratio_{V_2,V_1}(T_{m-2})
\end{align*}
We note that $\size(\bar f_{\alpha,a},V_2)=\alpha$. This means that $\power(\bar f_{\alpha,a},V_1)=\pi_{V_1}(T_{m-1})+\alpha\ratio_{T_2,T_1}(T_{m-2})<\pi_{V_1}(T_{m-1})+\alpha\ratio_{T_2,T_1}(P)=\power (f,V_1)$. This means that $\bar f_{\alpha,a}$ is a not a level-$\alpha$ UMP. The theorem follows after Lemma~\ref{lem:bordatest}.
\end{proof}

{\vspace{5mm}\noindent\bf Theorem~\ref{thm:mallowswinnerexist}.} Let $\mallows$ denote a Mallows' model with $n=1$ and any $m\ge 4$. There exists  $\epsilon>0$ such that for any $\varphi> 1-\epsilon$ and any $\alpha$, $\bar f_{\alpha,a}$ is a UMP test for  $H_0=(\ml(\ma)- H_1)$ vs.~$H_1=\atopr$.

\begin{proof} We first verify that when $K_\alpha=m-1$, $\bar f_{\alpha,a}$ is a UMP test.  For any $h_1\in H_1$, let $h_0^*\in H_0$ denote the ranking that is obtained from $h_1$ by moving $a$ down for one position. It is not hard to check that for any $V\in \ml(\ma)$, $\ratio_{h_0^*,h_1}(V)\le 1/\varphi$, and for all $V\in H_1$ we have $\ratio_{h_0^*,h_1}(V)= 1/\varphi$. This means that for any level-$\alpha$ test for $H_0$ vs.~$h_1$, the power cannot be more than $\alpha/\varphi$. We note that  $\bar f_{\alpha,a}$ is a level-$\alpha$ test whose power is exactly $\alpha/\varphi$. This means that for all $h_1\in H_1$, $\bar f_{\alpha,a}$ is a most powerful test for $H_0$ vs.~$h_1$. Therefore, when $K_\alpha=m-1$, $\bar f_{\alpha,a}$ is a UMP test.

For any $\alpha$ such that $K_\alpha\le m-2$, we will prove that for any $h_1\in H_1$, $\bar f_{\alpha,a}$ is a most powerful level-$\alpha$ test for $H_0$ vs.~$h_1$. This is done in the following steps. Step 1.~Find a least favorable distribution $\Lambda_\alpha^{h_1}$ whose support is the set of all rankings where $a$ is ranked at the second position. Step 2.~Verify that $\bar f_{\alpha,a}$ is the likelihood ratio test w.r.t.~$\Lambda_\alpha^{h_1}$, and step 3. verify that the two conditions in Lemma~\ref{lem:hvs} holds for $\Lambda_\alpha^{h_1}$. 
 
{\bf Step 1.} The main challenge is that in general there does not exist a uniformly least favorable distribution. For different $\alpha$ we define different $\Lambda_\alpha^{h_1}$ as follows. For any $\alpha$, we let $s_\alpha$ denote the smallest Borda score of the ranking $V$ such that $\bar f_{\alpha,a}(V)>0$. We have that $s_\alpha\le m-2$. Let the support of $\Lambda_\alpha^{h_1}$ be $T_{m-2}$, which is  the set of rankings where $a$ is ranked at the second position. We will solve the following system of linear equations to determine $\Lambda_\alpha^{h_1}$. For any $h_0^*\in T_{m-2}$ there is a variable $x[h_0,s_\alpha]$.
\begin{align*}\forall V\in T_{s_{\alpha}}, \sum_{h_0^*\in T_{m-2}} \ratio^{-1}_{h_0^*,h_1}(V)\cdot  x[h_0^*,s_\alpha]=m& &(\text{LP}_{s_\alpha}^{h_1})
\end{align*}
We note that as $\varphi\ra 1$, $\ratio^{-1}_{h_0^*,h_1}(V)=\frac{\pi_{h_0^*}(V)}{\pi_{h_1}(V)}=\varphi^{\kt(h_0^*,V)-\kt(h_1,V)}\ra 1$. Because there are $m$ variables and $m$ equations, as $\varphi\ra 1$ the solution to $\text{LP}_{s_\alpha}^{h_1}$ converges to $\vec 1$. Therefore, there exists $\epsilon>0$ such that for all $\varphi>1-\epsilon$, the linear systems $\{\text{LP}_s^{h_1}:s\le m-1, h_1\in H_1\}$ all have strictly positive solutions. Let $\{x^*[h_0^*,s_\alpha]|V\in T_{s_\alpha} \}$ denote a solution to $\text{LP}_{s_\alpha}^{h_1}$. For any $h_0^*\in T_{m-2}$, we let $\Lambda_\alpha^{h_1}(h_0^*)= \frac{x^*[h_0^*,s_\alpha]}{\sum_{h_0\in T_{m-2}} x^*[h_0,s_\alpha]}$.

{\bf Step 2.} To simplify notation we let $\lr_{\alpha}=\lr_{\alpha,\Lambda_\alpha^{h_1},h_1}$ denote the likelihood ratio test and let $\ratio=\ratio_{\Lambda_\alpha^{h_1},h_1}$ denote the likelihood ratio function w.r.t.~distribution $\Lambda_\alpha^{h_1}$ for $H_0$ vs.~$h_1$. To prove $\lr_\alpha=\bar f_{\alpha,a}$, we first prove that for any $V\in \ml(\ma)$ where $a$ is not ranked at the bottom position, $\ratio(V)>\ratio(\text{Down}^1_a(V))$, where we recall that $\text{Down}^1_a(V)$ is the ranking obtained from $V$ by moving $a$ down for one position.
\begin{align*}
&\frac{\sum_{h_0^*\in T_{m-2}}\Lambda_\alpha^{h_1}(h_0^*) \cdot \pi_{h_0^*}(\text{Down}_a^1(V))}{\sum_{h_0^*\in T_{m-2}}\Lambda_\alpha^{h_1}(h_0^*)\cdot \pi_{h_0^*}(V)}\\
=&\frac{\sum_{h_0^*\in T_{m-2}}\Lambda_\alpha^{h_1}(h_0^*) \cdot \varphi^{\kt(h_0^*,\text{Down}_a^1(V))}}{\sum_{h_0^*\in T_{m-2}}\Lambda_\alpha^{h_1}(h_0^*)\cdot \varphi^{\kt(h_0^*,V)}}\\
>&\frac{\sum_{h_0^*\in T_{m-2}}\Lambda_\alpha^{h_1}(h_0^*) \cdot \varphi^{\kt(h_0^*,V)}\cdot \varphi^{\kt(V,\text{Down}_a^1(V))}}{\sum_{h_0^*\in T_{m-2}}\Lambda_\alpha^{h_1}(h_0^*)\cdot \varphi^{\kt(h_0^*,V)}}\\
=&\varphi=\frac{\pi_{h_1}(\text{Down}_a^1(V))}{\pi_{h_1}(V)}
\end{align*}
The strict inequality holds because of (1) triangle inequality for Kentall-Tau distance, and (2) for any ranking $V$ where the top-ranked alternative in $h_0^*$ is ranked right below $a$, we have $\kt(h_0^*,V)+\kt(V,\text{Down}_a^1(V))>\kt(h_0^*,\text{Down}_a^1(V))$, and (3) for all $h_0^*\in T_{m-2}$, $\Lambda_\alpha^{h_1}(h_0^*)>0$.

It follows from the strict inequality that 
\begin{align*}
\ratio(V)=&\frac{\pi_{h_1}(V)}{\sum_{h_0^*\in T_{m-2}}\Lambda_\alpha^{h_1}(h_0^*)\cdot \pi_{h_0^*}(V)}\\
>&\frac{\pi_{h_1}(\text{Down}_a^1(V))}{\sum_{h_0^*\in T_{m-2}}\Lambda_\alpha^{h_1}(h_0^*) \cdot \pi_{h_0^*}(\text{Down}_a^1(V))}\\
=& \ratio(\text{Down}^1_a(V))
\end{align*}

Moreover, for any $V,V'\in T_{s_\alpha}$ we have $\ratio(V)=\ratio(V')$ by verifying $\text{LP}_{s_\alpha}^{h_1}$. Therefore, for any $V\in T_i$ with $i<s_\alpha$, we can move up the position of $a$ one by one until we reach the $(m-s_\alpha)$-th position. Let $V^*\in T_{s_\alpha}$ denote this ranking. It follows that $\ratio(V)<\ratio(V^*)$. Similarly for any $V'\in T_i$ with $i>s_\alpha$ we have $\ratio(V')>\ratio(V^*)$ for any $V^*\in T_{s_\alpha}$. This means that for any $V$ where $a$ is ranked above the $(m-s_\alpha)$-th position, we have $\lr_{\alpha}(V)=1$; for any  $V$ where $a$ is ranked below the $(m-s_\alpha)$-th position, we have $\lr_{\alpha}(V)=0$; for any  $V$ where $a$ is ranked at the $(m-s_\alpha)$-th position, we have that $\lr_{\alpha}(V)$ is the same and is between $0$ and $1$. It follows that $\lr_\alpha=\bar f_{\alpha,a}$.

{\bf Step 3.} Due to the symmetry $f_{\alpha,a}$ among alternatives in $\ma- \{a\}$, for any $i\le m-2$ and any $h_0,h_0'\in T_i$, we have $\size(\bar f_{\alpha,a},h_0)=\size(\bar f_{\alpha,a},h_0')$. Therefore, condition (i) in Lemma~\ref{lem:hvs} is satisfied. Choose arbitrary $h_0^{m-2}\in T_{m-2}$. For any $i\le m-3$, let $h_0^{i}\in T_i$ denote the ranking obtained from  $h_0^{i+1}$ by moving $a$ down for one position. To verify condition (ii) in Lemma~\ref{lem:hvs}, it suffices to prove that for any $i\le m-3$ and any $K\in \mathbb N$, we have 
\begin{equation}\label{eq:bordadom}
\begin{array}{c}\pi_{h_0^{m-2}}(\{V:\borda_a(V)\ge K\})\\ \ge \pi_{h_0^{i}}(\{V:\borda_a(V)\ge K\})\end{array}
\end{equation}
We will prove a slightly stronger lemma.

\begin{lem}\label{lem:bordadom}Under Mallows' model, for any $m$, any $\varphi$, any  $W\in\ml(\ma)$, any $b,c\in\ma$ such that $b\succ_W c$, and any $K$, we have
$\pi_{W}(\{V:\borda_b(V)\ge K\})\ge \pi_{W}(\{V:\borda_c(V)\ge K\})$.
\end{lem}

\begin{proof} The proof is similar to the proof of Lemma~\ref{lem:CC'}. It suffices to prove the lemma for the case where $b$ and $c$ are adjacent in $W$.  Let $\mcp=\{V\in\ml(\ma):\borda_b(V)\geq K\}$ and  $\mcp'=\{V\in\ml(\ma):\borda_c(V)\geq K\}$.  It follows that $\mcp\cap\mcp'$ is the set of rankings where both $b$ and $c$ are ranked  within top $m-K$ positions; $\mcp-\mcp'$ is the set of rankings where $b$ is ranked within top $m-K$ positions but $c$ is not; and $\mcp'-\mcp$ is the set of rankings where $c$ is ranked within top $m-K$ positions but $b$ is not. We let  $\mm$ be a permutation that switches $b$ and $c$. It is not hard to check that $\mm$ is a bijection between $(\mcp-\mcp')$ and $(\mcp'-\mcp)$, and because $b$ and $c$ are adjacent in $W$, for any $V\in \mcp$, we have $\kt(M(V),W)=\kt(V,W)+1$, which means that $\pi_W(V)=\pi(M(V))/\varphi$. Therefore, we have

\begin{align*}
&\pi_{W}(\{V:\borda_b(V)\ge K\})-\pi_{W}(\{V:\borda_c(V)\ge K\})\\
=& \pi_W(\mcp)-\pi_W(\mcp')= \pi_W(\mcp-\mcp')-\pi_W(\mcp'-\mcp)\\
=&\pi_W(\mcp-\mcp')-\pi_W(M(\mcp-\mcp'))\\
=&(\frac{1}{\varphi}-1)\pi_W(\mcp-\mcp') \ge 0
\end{align*}
This proves the lemma. 
\end{proof}

Let $W$ be an arbitrary ranking and let $M_i$ denote a permutation such that $M_i(h_0^i)=W$. We have $\pi_{h_0^{i}}(\{V:\borda_a(V)\ge K\})=\pi_{M_i(h_0^{i})}(\{V:\borda_{M_i(a)}(V)\ge K\})$. We note that $M_i(a)$ is the alternative that is ranked at the $(m-i)$-th position in $W$. Inequality (\ref{eq:bordadom}) follows after applying Lemma~\ref{lem:bordadom}. This means that condition (ii) in Lemma~\ref{lem:hvs} is also satisfied. Therefore, by Lemma~\ref{lem:hvs}, $\bar f_{\alpha,a}$ is a level-$\alpha$ most powerful test for $H_0$ vs.~$h_1$. Since $\bar f_{\alpha,a}$ does not depend on $h_1$, it is a level-$\alpha$ UMP test for $H_0$ vs.~$H_1$.
\end{proof}

{\vspace{5mm}\noindent\bf Lemma~\ref{lem:ind}.}
For any $\mm_X$ and $\mm_Y$, suppose $\Lambda_X$ is a least favorable distribution for composite vs.~simple test ($H_{0,X}$ vs.~$x_1$) under $\mm_X$. Given $y_1\in \Theta_Y$, let $\Lambda^*$ be the distribution over $H_{0,X}\times\Theta_Y$ where for all $x\in H_{0,X}$, $\Lambda^*(x,y_1)=\Lambda_X(x)$. Then $\Lambda^*$ is a least favorable distribution for $H_{0,X}\times \Theta_Y$ vs.~$(x_1,y_1)$ under $\mm_X\otimes \mm_Y$.

\begin{proof} Let $x_0^1,\ldots,x_0^K\in \Theta_X$ denote the support of $\Lambda_X$.  The theorem is proved by applying Lemma~\ref{lem:hvs}. For any $0<\alpha< 1$ and any $P=(P_X,P_Y)\in\ms_X\times\ms_Y$, we have the following calculation. In this proof $\ratio$ stands for $\ratio_{\Lambda^*,(x_1,y_1)}$ and $\lr_{\alpha}$ stands for $\lr_{\alpha,\Lambda^*,(x_1,y_1)}$.
\begin{align*}
\ratio(P_X,P_Y)
=&\frac{\pi_{x_1,y_1}(P)}{\sum_{k=1}^K\Lambda^*(x_0^k,y_1)\pi_{(x_0^k,y_1)}(P)}\\
=&\frac{\pi_{x_1}(P_X)\cdot \pi_{y_1}(P_Y)}{\sum_{k=1}^K\Lambda^*(x^k_{0},y_1)\pi_{x^k_{0}}(P_X)\cdot\pi_{y_1}(P_Y)}\\
=&\frac{\pi_{x_1}(P_X)}{\sum_{k=1}^K\Lambda(x_0^k)\pi^k_{x_0}(P_X)}
=\ratio_{\Lambda,x_1}(P_X)
\end{align*}
It follows that for any pair of samples $(P_X,P_Y), (P_X',P_Y')\in\ms_X\times \ms_Y$, $\ratio(P_X,P_Y)\ge \ratio(P_X',P_Y')$ if and only if $\ratio_{\Lambda,x}(P_X)\ge \ratio_{\Lambda,x}(P_X')$. This means that for any $(P_X,P_Y)$, $\lr_{\alpha}(P_X,P_Y)=\lr_{\alpha,\Lambda,x_1}(P_X)$. Therefore, for any $x_0\in H_{0,X}$, we have
\begin{align*}
&\size(\lr_\alpha,(x_0,y_1))\\
=&\sum_{(P_X,P_Y)\in\ms_X\times\ms_Y} \pi_{x_0}(P_X) \pi_{y_1}(P_Y)\lr_\alpha(P_X,P_Y)\\
=&\sum_{(P_X,P_Y)\in\ms_X\times\ms_Y} \pi_{x_0}(P_X) \pi_{y_1}(P_Y)\lr_{\alpha,\Lambda,x_1}(P_X)\\
=&\sum_{P_X\in\ms_X} \pi_{x_0}(P_X)\lr_{\alpha,\Lambda,x_1}(P_X)\\
=&\size(\lr_{\alpha,\Lambda,x_1},x_0)
\end{align*}
Therefore, by Lemma~\ref{lem:hvs}, for any $(x_0^*,y_1)\in\supp(\Lambda^*)$, we have $\size(\lr_\alpha,(x_0,y_1))=\size(\lr_{\alpha,\Lambda,x_1},x_0)=\alpha$ because $x_0^*\in \supp(\Lambda)$; for any $(x_0,y)\in H_{0,X}\times \Theta_Y$, we have $\size(\lr_\alpha,(x_0,y))=\size(\lr_{\alpha,\Lambda,x_1},x_0)\le \alpha$. This means that the two conditions in Lemma~\ref{lem:hvs} are satisfies, which proves the theorem.
\end{proof}

{\vspace{5mm}\noindent\bf Lemma~\ref{lem:extdomain}.} {\em For any model $\mm_X$ and any $t\in\mathbb N$, suppose $\Lambda$ is a uniformly least favorable distribution for composite vs.~simple test ($H_0$ vs.~$h_1$) under $\mm_X$. Then $\ext(\Lambda,h_1,t)$ is a uniformly least favorable distribution for $\ext(H_0,h_1,t)$ vs.~$\vec h_1$ in $(\mm_X)^t$.}

\begin{proof} Again the proof is done by applying Lemma~\ref{lem:hvs}. 
We first prove a claim that characterizes samples whose likelihood ratio is no more than a given threshold. To this end, it is convenient to use the inverse of the likelihood ratio. To simplify notation, in this proof we let $\Lambda^*=\ext(\Lambda,h_1,t)$, let $H_0^*=\ext(H_0,h_1,t)$, let $\lr_\alpha=\lr_{\alpha,\Lambda^*,\vec h_1}$, $\ratio=\ratio_{\Lambda^*,\vec h_1}$. 

\begin{claim}\label{claim:invratio} For any $k_{\alpha}$ and any $\vec x\in \ms^t$, $\sum_{j=1}^t\ratio^{-1}_{\Lambda, h_1}(x_j)=t\cdot \ratio^{-1}(\vec x)$.
\end{claim}
\begin{proof} we have
$\ratio^{-1}(\vec x)=\frac{1}{t}\cdot \frac{\sum_{j=1}^t\sum_{h_0\in H_0}\Lambda(h_0)\cdot \pi_{(h_0,[\vec h_1]_{-j})}(\vec x)}{\pi_{\vec h_1}(\vec x)}$\\
$
=\frac{1}{t}\cdot \frac{\sum_{j=1}^t\sum_{h_0\in H_0}\Lambda(h_0)\cdot \pi_{h_0}(x_j)\cdot \pi_{[\vec h_1]_{-j}}(x_j)}{\pi_{h_1}(x_j)\cdot \pi_{[\vec h_1]_{-j}}(x_j)}$\\ 
$
=\frac{1}{t}\sum_{j=1}^t\ratio^{-1}_{\Lambda,h_1}(x_j)$ 
\end{proof}

The next lemma proves the following:
For any $\vec z\in H_0^*$ and any $j\le t$, suppose the $j$-th component is not in $\supp(\Lambda)\cup\{h_1\}$. If we fix all components except $j$-th in $\vec z$ and change the $j$-th component to $h_0^*\in \supp(\Lambda)$, then the size of $\lr_\alpha$ will increase. If we further change the $j$-th component to $h_1$, then the size of $\lr_\alpha$ will further increase. 

\begin{lem}\label{lem:change} For  any $0\le \alpha\le 1$, any $j\le t$, any $\vec z_{-j}\in \Theta^{t-1}$, any $h_0\in H_0$, and any $h_0^*\in \supp(\Lambda)$, we have $\size(\lr_\alpha,(h_0,\vec z_{-j}))\le \size(\lr_\alpha,(h_0^*,\vec z_{-j}))\le \size(\lr_\alpha,(h_1,\vec z_{-j}))$.
\end{lem}
\begin{proof} For any $\vec z_{-j}\in \Theta^{n-1}$, we have 
\begin{align*}
&\size(\lr_\alpha,(h_0,\vec z_{-j}))=\pi_{(h_0,\vec z_{-j})}(\{\vec x\in\ms^t:\ratio(\vec x)> k_\alpha^*\})\\
&+\gamma_\alpha^* \pi_{(h_0,\vec z_{-j})}(\{\vec x\in\ms^t:\ratio(\vec x)= k_\alpha^*\})
\end{align*}
For any $\vec x$, we let $\text{Sum}(\vec x)=\sum_{l=1}^t\ratio^{-1}_{\Lambda,h_1}(x_l)$ and for any $j\le t$, we let $\text{Sum}(\vec x_{-j})=\sum_{l\ne j}\ratio^{-1}_{\Lambda,h_1}(x_l)$. By Claim~\ref{claim:invratio}, we have 
\begin{align*}
&\pi_{(h_0,\vec z_{-j})}(\{\vec x\in\ms^t:\ratio(\vec x)>k_\alpha^*\})\\
=&\pi_{(h_0,\vec z_{-j})}(\{\vec x\in\ms^t:\text{Sum}(\vec x)< t/k_\alpha^*\})\\
=&\pi_{(h_0,\vec z_{-j})}(\{\vec x\in\ms^t:\text{Sum}(\vec x_{-j})+\ratio^{-1}_{\Lambda,h_1}(x_j)< t/k_\alpha^*\})\\
=&\int_{0}^{t/k_\alpha^*}\sum\nolimits_{\vec x_{-j}\in\ms^{t-1}:\text{Sum}(\vec x_{-j})=p}\\
&\sum\nolimits_{x_j:\ratio^{-1}_{\Lambda,h_1}(x_j)<t/k_\alpha^*-p}\pi_{(h_0,\vec z_{-j})}(\vec x) dp\\
=&\int_{0}^{t/k_\alpha^*}\pi_{\vec z_{-j}}(\{\vec x_{-j}\in\ms^{t-1}:\text{Sum}(\vec x_{-j})=p\})\\
&\cdot\pi_{h_0}(\{x_j:\ratio^{-1}_{\Lambda,h_1}(x_j)<t/k_\alpha^*-p\}) dp\\
=&\int_{0}^{t/k_\alpha^*}Q(\vec z_{-j},p)\cdot\pi_{h_0}(\{x_j:\ratio^{-1}_{\Lambda,h_1}(x_j)<t/k_\alpha^*-p\}) dp\\
\end{align*}
where $Q(\vec z_{-j},p)=\pi_{\vec z_{-j}}(\{\vec x_{-j}\in\ms^{t-1}:\text{Sum}(\vec x_{-j})=p\})$.
Given $p$ and $\gamma_\alpha^*$, let $\alpha'$ denote the size of the likelihood ratio test $\lr_{\alpha',\Lambda,h_1}$, where the threshold $k_{\alpha'}$ is $1/(t/k_\alpha^*-p)$ and $\gamma_{\alpha'}=\gamma_\alpha^*$. We have

\begin{equation}\label{eq:sizeint}
\size(\lr_\alpha,(h_0,\vec z_{-j}))
=\int_{0}^{t/k_\alpha^*}Q(\vec z_{-j},p)\cdot\size(\lr_{\alpha',\Lambda,h_1},h_0)dp
\end{equation}
We note that in Equation (\ref{eq:sizeint}), $\alpha'$ is a function of $t$, $p$, $k_{\alpha}^*$, and $\gamma_{\alpha}^*$. Because $\Lambda$ is a uniformly least favorable distribution, it follows from Lemma~\ref{lem:hvs} that for any $h_0^*\in\supp(\Lambda)$ and any $h_0\in (H_0-\supp(\Lambda))$, we have
$$\size(\lr_{\alpha',\Lambda,h_1},h_0)\le \alpha'\le \size(\lr_{\alpha',\Lambda,h_1},h_0^*)$$

Then by Equation~(\ref{eq:sizeint}), for any $h_0\in (H_0- \supp(\Lambda))$ and any $h_0^*\in \supp(\Lambda)$, we have
\begin{align*}
&\size(\lr_\alpha,(h_0,\vec z_{-j}))\\
=&\int_{0}^{t/k_\alpha^*}Q(\vec z_{-j},p)\cdot\size(\lr_{\alpha',\Lambda,h_1},h_0)dp\\
\le& \int_{0}^{t/k_\alpha^*}Q(\vec z_{-j},p)\cdot\size(\lr_{\alpha',\Lambda,h_1},h_0^*)dp\\
=&\size(\lr_\alpha,(h_0^*,\vec z_{-j}))
\end{align*}

To prove the last inequality in the lemma, we prove a claim that holds for any least favorable distribution and the corresponding likelihood ratio test. The $\size(\cdot)$ function in the claim is extended to $h_1\in H_1$ in the natural way.

\begin{claim}\label{claim:sizeh1} For any model, any composite vs.~simple test ($H_0$ vs.~$h_1$), suppose $\Lambda$ is a level-$\eta$ least favorable distribution. Then we have 
$\size(\lr_{\eta},h_1)\ge \eta=\size(\lr_{\eta},h_0^{\Lambda})$. \footnote{We recall that $h_0^\Lambda$ is the combined $H_0$ by $\Lambda$.}
\end{claim}
\begin{proof} For the sake of contradiction suppose this is not true, that is, for any $h_0^*\in\supp(\Lambda)$ we have $\size(\lr_{\eta},h_1)< \eta=\size(\lr_{\eta},h_0^*)$. It follows that $k_{\eta}\le1$, otherwise we have
 
\begin{align*}
&\size(\lr_\eta,h_1)\\
=&\sum_{P\in\ms:\ratio(P)>k_\eta}\pi_{h_1}(P)+\gamma_\eta\sum_{P\in\ms:\ratio(P)=k_\eta}\pi_{h_1}(P)\\
\ge&\sum_{P\in\ms:\ratio(P)>k_\eta}\pi_{\Lambda}(P)
\cdot k_\eta+\gamma_\eta\sum_{P\in\ms:\ratio(P)=k_\eta}\pi_{\Lambda}(P)
\cdot k_\eta\\
>&\sum_{P\in\ms:\ratio(P)>k_\eta}\pi_{\Lambda}(P)+\gamma_\eta\sum_{P\in\ms:\ratio(P)=k_\eta}\pi_{\Lambda}(P)
=\eta,
\end{align*} 
which is a contradiction. Therefore, we have 

\begin{align*}1\\
=&\size(\lr_\eta,h_1)+\sum\nolimits_{P\in \ms:\ratio(P)<k_\eta}\pi_{h_1}(P)\\
&+(1-\gamma_\eta)\sum\nolimits_{P\in \ms:\ratio(P)=k_\eta}\pi_{h_1}(P)\\
<&\eta+\sum\nolimits_{P\in \ms:\ratio(P)<k_\eta}\pi_{\Lambda}(P)\cdot k_\eta\\
&+(1-\gamma_\eta)\sum\nolimits_{P\in \ms:\ratio(P)=k_\eta}\pi_{\Lambda}(P)\cdot k_\eta\\
\le &\eta+k_\eta(1-\size(\lr_{\eta},h_0^\Lambda))\le 1,
\end{align*}
which is a contradiction.
\end{proof}

Applying Claim~\ref{claim:sizeh1} to $\lr_{\alpha',\Lambda,h_1}$, we have
\begin{align*}
&\size(\lr_\alpha,(h_0^*,\vec z_{-j}))\\
=&\int_{0}^{t/k_\alpha^*}Q(\vec z_{-j},p)\cdot\size(\lr_{\alpha',\Lambda,h_1},h_0^*)dp\\
\le& \int_{0}^{t/k_\alpha^*}Q(\vec z_{-j},p)\cdot\size(\lr_{\alpha',\Lambda,h_1},h_1)dp\\
=&\size(\lr_\alpha,(h_1,\vec z_{-j}))
\end{align*}

This finishes the proof of Lemma~\ref{lem:change}.
\end{proof}

It follows from Lemma~\ref{lem:change} that for any $j\le t$  and any $h_0^*\in \supp(\Lambda)$, we have that $\size(\lr_\alpha,(h_0^*,[\vec h_1]_{-j}))$ is the same. Due to symmetry, for any $\vec h_0^*\in H_0^*$, $\size(\lr_\alpha,h_0^*)$ is the same and is therefore equivalent to $\alpha$. This verifies condition (i) in Lemma~\ref{lem:hvs}. 

Condition (ii) in Lemma~\ref{lem:hvs} is verified by recursively applying Lemma~\ref{lem:change}. Given any $\vec h_0\in H_0^*- \supp(\Lambda^*)$, there must exist $j\le t$ such that $[\vec h_0]_j\ne h_1$. We then change $[\vec h_0]_j$ to an arbitrary $h_0^*\in \supp(\Lambda)$, then change the other components of $\vec h_0$ to $h_1$ one by one. Each time we make the change the size of $\lr_\alpha$ does not decrease according to Lemma~\ref{lem:change}. At the end of the process we obtain $(h_0^*,[\vec h_1]_j)\in\supp(\Lambda^*)$, at which the size of $\lr_\alpha$ is $\alpha$. The theorem follows after applying Lemma~\ref{lem:hvs}. 
\end{proof}

We now define a test $\bar f_{\alpha,a}$ for $H_0=(\ml(\ma)- H_1)$ vs.~$H_1=\atopr$ and prove that if a UMP test exists, then $\bar f_{\alpha,a}$ must also be a UMP test. For any $V\in \ml(\ma)$ and any alternative $a\in\ma$, we let $\borda_a(V)$ denote the Borda score of $a$ in $V$. That is, $\borda_a(V)$ is the number of alternatives that are ranked below $a$ in $V$. For any $V\in\ml(\ma)$, we let $\bar f_{\alpha,a}(V)=\left\{\begin{array}{ll}1&\text{if }\borda_a(V)>K_\alpha\\
0&\text{if }\borda_a(V)<K_\alpha\\ 
\Gamma_\alpha&\text{if }\borda_a(V)=K_\alpha
\end{array}\right.,$
where $K_\alpha$ and $\Gamma_\alpha$ are chosen so that the size of $\bar f_{\alpha,a}$ is $\alpha$. In other words, $\bar f_{\alpha,a}$ calculates the Borda score of $a$ in the input profile, and if it is larger than a threshold $K_\alpha$ then $H_0$ is rejected. It is not hard to see that $\bar f_{\alpha,a}$ equals to $f_{\alpha',a}$ with a possibly different level $\alpha'$ (defined in Theorem~\ref{thm:winnerbottomtopm}).

\begin{lem}\label{lem:bordatest} If there exists a level-$\alpha$ UMP test for $H_0=(\ml(\ma)- H_1)$ vs.~$H_1=\atopr$, then $\bar f_{\alpha,a}$ is also a level-$\alpha$ UMP test.
\end{lem}
\begin{proof} Let $f_\alpha$ denote a level-$\alpha$ UMP test. For any permutation $M$ over $\ma-\{a\}$, we let $M(f_{\alpha})$ denote the test such that for any $V\in\ml(\ma)$, $M(f_{\alpha})(V)=f_\alpha(M(V))$. Because the Kendall-Tau distance is invariant to permutations, we have that for any $h_0\in H_0$, $\size(f_{\alpha},h_0)= \size(M(f_{\alpha}),M(h_0))$, and for any $h_1\in H_1$, $\power(f_{\alpha},h_1)=\power(M(f_{\alpha}),M(h_1))$. Therefore $\size(M(f_\alpha))=\alpha$. Also because the multi-set of $\{\power(f_\alpha,h_1):h_1\in H_1\}$ is the same as the multi-set $\{\power(M(f_\alpha),h_1):h_1\in H_1\}$, for all $h_1\in H_1$, we must have $\power(f_\alpha,h_1)=\power(M(f_\alpha),h_1)$, otherwise there exists $h_1\in H_1$ such that $\power(f_\alpha,h_1)<\power(M(f_\alpha),h_1)$, which contradicts the assumption that $f_\alpha$ is UMP.

It follows that for any permutation $M$ over $\ma-\{a\}$, $M(f_\alpha)$ is also UMP. Therefore, $\bar f_\alpha=\frac{1}{(m-1)!}\sum_{M} M(f_\alpha)$ is also UMP. We note that for any $V,V'$ where $a$ has the same Borda score, there exists a permutation $M$ over $\ma-\{a\}$ so that $M(V)=V'$. This means that $\bar f_\alpha(V)=\bar f_\alpha(V')$. 

We now prove that $\bar f_\alpha$ must be $\bar f_{\alpha,a}$ as in the statement of the Lemma. More precisely, we will prove that for any $V,V'$ such that $\borda_a(V)>\borda_a(V')$, if $\bar f_\alpha(V')>0$ then $\bar f_\alpha(V)=1$. Suppose for the sake of contradiction that this is not true, and there exist $V,V'$ such that $s_1=\borda_a(V)>\borda_a(V')=s_2$, $\bar f_\alpha(V')>0$, and $\bar f_\alpha(V)<1$. For any $s\le m-1$, we let $T_{s}$ denote the set of rankings where the Borda score of $a$ is $s$. That is, $T_s=\{V\in\ml(\ma):\borda_a(V)=s\}$. We will prove that for any $s_1>s_2$, $T_{s_1}$ as a whole is more ``cost effective'' than $T_{s_2}$ as a whole for any $h_0\in H_0$ against any $h_1\in H_1$. More precisely, we will prove that $\ratio_{h_0,h_1}(T_{s_1})>\ratio_{h_0,h_1}(T_{s_2})$.

For any $s\le m-2$ and any $h_0\in T_s$, let $h_1$ denote the ranking in $T_{m-1}=H_1$ that is obtained from $\theta$ by raising $a$ to the top position. For any $V_{s_1}\in T_{s_1}$, we let $\text{Down}_{a}^{s_1-s_2}(V_{s_1})\in T_{s_2}$ denote the ranking that is obtained from $V_{s_1}$ by moving $a$ down for $s_1-s_2$ positions, that is, from the $(m-s_1)$-th position to the $(m-s_2)$-th position. We have
\begin{align*}
&\frac{ \pi_{h_0}(T_{s_2})}{\pi_{h_0}(T_{s_1})}\\
= &\frac{\sum_{V\in T_{s_2}}\pi_{h_0}(V)}{\sum_{V\in T_{s_1}}\pi_{h_0}(V)}
=\frac{\sum_{V\in T_{s_1}}\pi_{h_0}(\text{Down}^{s_1-s_2}_a(V))}{\sum_{V\in T_{s_1}}\pi_{h_0}(V)}\\
=&\frac{\sum_{V\in T_{s_1}}\varphi^{\kt(h_0,\text{Down}^{s_1-s_2}_a(V))}}{\sum_{V\in T_{s_1}}\varphi^{\kt(h_0,V)}}\\
> &\frac{\sum_{V\in T_{s_1}}\varphi^{\kt(h_0,V)}\cdot \varphi^{\kt(V,\text{Down}^{s_1-s_2}_a(V))}}{\sum_{V\in T_{s_1}}\varphi^{\kt(h_0,V)}}\\
=&\varphi^{s_1-s_2}=\frac{ \pi_{h_1}(T_{s_2})}{\pi_{h_1}(T_{s_1})}
\end{align*}
The inequality is due to triangle inequality for Kendall-Tau distance. It is strict because for any $V\in T_{s_1}$ where the top-ranked alternative in $h_0$ is ranked between the $(m-s_1)$-th and $(m-s_2)$-th position, $\kt(h_0,\text{Down}^{s_1-s_2}_a(V))<\kt(h_0,V)+\kt (V,\text{Down}^{s_1-s_2}_a(V))$. Therefore, $\frac{ \pi_{h_0}(T_{s_2})}{\pi_{h_0}(T_{s_1})}> \frac{\pi_{h_1}(T_{s_2})}{\pi_{h_1}(T_{s_1})}$, which means that $\ratio_{h_0,h_1}(T_{s_1})=\frac{\pi_{h_1}(T_{s_1})}{\pi_{h_0}(T_{s_1})}> \frac{\pi_{h_1}(T_{s_2})}{ \pi_{h_0}(T_{s_2})}=\ratio_{h_0,h_1}(T_{s_2})$.

Therefore, we can find sufficiently small $\epsilon,\delta>0$, and replace $\epsilon T_{s_2}$ by $\delta T_{s_1}$ without changing the size. This will increase the power of $\bar f_\alpha$ because $T_{s_1}$ is strictly more cost effective than $T_{s_2}$, which contradicts the assumption that $\bar f_\alpha$ is a UMP test. Therefore, $\bar f_\alpha =\bar f_{\alpha,a}$, which proves the lemma.  
\end{proof}

\end{document}